\documentclass[11pt]{amsart}

\usepackage[T1]{fontenc}
\usepackage{lmodern}
\usepackage{microtype}
\usepackage{amsmath,amssymb,amsthm,mathtools}
\usepackage{booktabs}
\usepackage{array}
\usepackage{enumitem}
\usepackage[colorlinks=true,linkcolor=blue,citecolor=blue,urlcolor=blue,
  pdftitle={Jacobi descent charts and logarithmic quotient coordinates for split symmetric spaces},
  pdfauthor={Jonathan Sparling}]{hyperref}
\usepackage[capitalise,noabbrev]{cleveref}
\usepackage{xcolor}
\usepackage{aliascnt}

\newtheorem{theorem}{Theorem}[section]
\newaliascnt{proposition}{theorem}
\newtheorem{proposition}[proposition]{Proposition}
\aliascntresetthe{proposition}
\newaliascnt{lemma}{theorem}
\newtheorem{lemma}[lemma]{Lemma}
\aliascntresetthe{lemma}
\newaliascnt{corollary}{theorem}
\newtheorem{corollary}[corollary]{Corollary}
\aliascntresetthe{corollary}
\newaliascnt{conjecture}{theorem}

\aliascntresetthe{conjecture}
\newaliascnt{question}{theorem}

\aliascntresetthe{question}
\theoremstyle{definition}
\newaliascnt{definition}{theorem}
\newtheorem{definition}[definition]{Definition}
\aliascntresetthe{definition}
\newaliascnt{example}{theorem}
\newtheorem{example}[example]{Example}
\aliascntresetthe{example}
\theoremstyle{remark}
\newaliascnt{remark}{theorem}
\newtheorem{remark}[remark]{Remark}
\aliascntresetthe{remark}
\newaliascnt{warning}{theorem}
\newtheorem{warning}[warning]{Warning}
\aliascntresetthe{warning}

\newcommand{\Fb}{\overline F}

\newcommand{\cQ}{\mathcal Q}
\newcommand{\cB}{\mathcal B}

\newcommand{\cS}{\mathcal S}
\newcommand{\A}{\mathbb A}
\newcommand{\Gm}{\mathbb G_m}
\newcommand{\Lie}{\operatorname{Lie}}
\newcommand{\Ad}{\operatorname{Ad}}
\newcommand{\ad}{\operatorname{ad}}
\newcommand{\diag}{\operatorname{diag}}
\newcommand{\rank}{\operatorname{rank}}

\newcommand{\GL}{\operatorname{GL}}
\newcommand{\SL}{\operatorname{SL}}
\newcommand{\Sp}{\operatorname{Sp}}

\newcommand{\Hom}{\operatorname{Hom}}

\newcommand{\dmeas}{\,d^{\times}}
\newcommand{\val}{\operatorname{val}}
\newcommand{\cC}{\mathfrak c}
\newcommand{\abs}[1]{\lvert #1\rvert}
\newcommand{\trans}{\mathsf t}
\newcommand{\g}{\mathfrak g}
\newcommand{\h}{\mathfrak h}
\newcommand{\p}{\mathfrak p}
\newcommand{\aA}{\mathfrak a}
\newcommand{\tT}{\mathfrak t}
\newcommand{\mM}{\mathfrak m}

\newcommand{\slie}{\mathfrak{sl}}
\newcommand{\glie}{\mathfrak{gl}}
\newcommand{\splie}{\mathfrak{sp}}
\newcommand{\solie}{\mathfrak{so}}

\title[Jacobi descent charts]
{Jacobi descent charts and logarithmic quotient coordinates\\
for split symmetric spaces}
\author{Jonathan Sparling}
\email{jon@liuva.ai}
\date{August 2026}

\begin{document}

\begin{abstract}
Let $F$ be a non-Archimedean local field of characteristic zero and residue characteristic different from $2$, and let $(G,H)$ be an equal-rank split symmetric pair with $G$ split semisimple and adjoint and with an $F$-split maximal $\theta$-split torus. Motivated by the singular contribution near the nilpotent fiber on the $\theta$-fixed side of the infinitesimal local relative trace formula, we construct finite descent charts on which families of regular elements in $\mathfrak{h}$ are conjugate, via Cartan lifts defined over $F$, to an everywhere-regular Jacobi family specializing to a fixed principal nilpotent element of $\mathfrak{g}$.

The construction begins with the sparse Jacobi family
\[
Q(q)=\sum_{\alpha\in\Delta}
\bigl(\gamma_\alpha n_{-\alpha}+q_\alpha n_\alpha\bigr)
\]
which is regular for every $q$, including the principal nilpotent value $Q(0)$. The induced map to the adjoint quotient has generic rank equal to the number of odd exponents of $G$. For split symmetric pairs this number is $\operatorname{rank}H$, so in equal rank the Jacobi family gives a generically finite \'etale quotient chart. Kummer theory classifies the twisted square-root covers on which the required Cartan conjugation is defined over $F$, and these covers exhaust the rational branches.

On every such branch, with square-root coordinates $q_\alpha=z_\alpha^2$, the quotient Jacobian cancels the relative Weyl discriminant exactly:
\[
\frac{ds}{\abs{D_H^G}^{1/2}}
=
C\prod_{\alpha\in\Delta}\dmeas z_\alpha
\]
Thus the normalized $H^\circ$-quotient density becomes a multiplicative Haar measure on the parameter torus, uniformly across all rational twists. Moreover, after a finite clopen refinement, the associated Iwasawa heights are piecewise affine functions of the valuations $\operatorname{val}(z_\alpha)$. For the symmetric pair $(\operatorname{Sp}_{2n},\operatorname{GL}_n)$, we realize the construction explicitly in Hurwitz coordinates and factor the Cartan lift through commuting long-root $\operatorname{SL}_2$-subgroups.
\end{abstract}

\maketitle
\tableofcontents

\section{Introduction}

Let $G$ be a split connected reductive group over a non-Archimedean local
field $F$ of characteristic zero and residue characteristic different from
$2$, let $\theta$ be an involution, let $H=G^\theta$, and write
\[
 \g=\h\oplus\p
\]
for the eigenspace decomposition.  The local problem studied here is to
construct rational regular families in $\h$ which approach the nilpotent
fiber in a controlled direction, while retaining enough information to track
both invariant measures and Iwasawa heights.  This problem is intrinsic to
the symmetric pair; its motivating occurrence is on the $\theta$-fixed side of
an infinitesimal local relative trace formula.

In that setting one starts from the identity
\begin{equation}\label{eq:intro-plancherel}
 \int_{H\backslash G}\phi(g)
     \int_{\p} f(g^{-1}Xg)\,dX\,dg
 =
 \int_{H\backslash G}\phi(g)
     \int_{\h}\widehat f(g^{-1}Xg)\,dX\,dg
\end{equation}
where $\phi=\phi_\mu$ is a truncation function on $H\backslash G$, introduced to
guarantee convergence, and $\widehat f$ denotes the Fourier transform with
respect to a fixed invariant pairing and compatible self-dual measures.

The left-hand side is the $\theta$-split side.  In
\cite{SparlingThetaSplit} it was expanded by Weyl integration on $\p$ into
weighted semisimple orbital integrals.  For a Levi subgroup $M$, the resulting
relative weight is eventually the lattice count
\begin{equation}\label{eq:theta-split-weight}
 \omega_M(g,\mu)
 =\#\left\{\nu\in X_*(A_M)^-\cap\operatorname{Im}\tau:
 \nu\in
 \operatorname{Hull}
 \{\mu_B-H_B(g)+H_{\bar B}(\theta(g))\}^{*}\right\}
\end{equation}
where
\[
 \tau(x)=\theta(x)^{-1}x
\]
A toric Euler--Poincar\'e argument gives a polynomial approximation
$J^-(f,\mu)$ to that side.  In particular, for
$\mu=\mu_1+d\mu_2$ in a sufficiently regular direction $\mu_2$,
\begin{equation}\label{eq:theta-split-limit}
 \lim_{d\to\infty}
 \bigl(\text{$\theta$-split side}-J^-(f,\mu)\bigr)=0
\end{equation}
No nilpotent correction of the kind studied below occurs in this
$\theta$-split expansion.

The right-hand side of \eqref{eq:intro-plancherel} is the $\theta$-fixed
side.  Its difficulty is different.  The most direct orbital expansion on
$\h$ may become singular near the origin, and in general one must perform a second
truncation procedure there.  The complete rank-one calculation for
$F^\times\backslash\SL_2(F)$ in \cite{SparlingRankOne} makes the
distinction manifest.  The $\theta$-split side is a sum of semisimple
weighted orbital integrals indexed by rational torus classes, whereas the
$\theta$-fixed side contains a weighted regular nilpotent integral of the
form
\[
 \int_F \widetilde f
 \begin{pmatrix}0&2n\\0&0\end{pmatrix}
 \bigl(1+\val(n)+\mu-\val(\epsilon)\bigr)\,dn
\]
together with an integral over the complement of a small neighborhood of the
origin in $\h$.  The polynomial weight in this nilpotent term is created by
the second truncation procedure.

The present paper develops the local geometric mechanism behind this
$\theta$-fixed calculation.  After the fixed-side expansion one is led to a
family of regular elements of $\h$ approaching the nilpotent fiber.  The
second truncation removes a small neighborhood of that fiber in the invariant
quotient.  To analyze its contribution, one seeks to pull the neighborhood
back to descent charts on finite covers of the regular locus, conjugate the
corresponding regular family toward a fixed regular nilpotent element of
$\g$, and then change the order of integration.  The eventual nilpotent
weight is therefore not an input: it is expected to be produced by
integrating the original truncation function over the quotient neighborhood.

The geometric problem addressed here is to construct these families over
$F$, identify their rational branches, and find toroidal parameters on finite
covers in which both the quotient measure and the resulting Iwasawa-height
variables can be computed.
On the square-root charts used below, the descended family in $\h$ collapses
linearly to zero, while its inverse conjugate extends to a principal nilpotent
element of $\g$.  The singular Cartan conjugation exposing this nilpotent limit
is also what introduces the additional valuation variables in the fixed-side
weight.

The existence of a polynomial trace formula was an important motivation for
this calculation.  In the rank-one formula, the polynomial approximations to
the two sides agree because their difference tends to zero along an
unbounded truncation ray.  In higher rank, the construction below isolates
the new dependence created by the nilpotent degeneration: it produces
explicit valuation variables for the conjugating family and retains the root
and Galois data carried by the degeneration.

Recent work makes this local problem particularly relevant.  Xue proved
relative Shalika-germ expansions for
$\GL_n\times\GL_n\backslash\GL_{2n}$ \cite{Xue}.  Li constructed local
trace formulas for infinitesimal Guo--Jacquet symmetric spaces and proved
representability and Fourier-transform identities for weighted orbital
integrals \cite{LiLocal,LiFourier}.  Chaudouard and Li reduce nonregular
geometric terms in a global relative trace formula, by descent to
centralizers, to nilpotent contributions of infinitesimal trace formulas
\cite{ChaudouardLi}.  Mahendraker develops a relative trace formula for
Galois periods \cite{Mahendraker}.  In the group case, Chen's
Arthur--Shalika germs provide a useful comparison for the homogeneity and
parabolic descent of weighted orbital integrals \cite{ChenWeighted}.  The
fixed-side operation considered here is different:
the weight attached to the nilpotent contribution is created by integrating a
small quotient neighborhood.  These developments also highlight a concrete local task:
construct rational regular families approaching the relevant nilpotent orbits
and coordinates in which that neighborhood integration can be carried out.

The main ingredients of the paper are as follows.
\begin{enumerate}[label=\textup{(\roman*)},leftmargin=2.5em]
\item For semisimple $G$, the extremal identity
      \[
       \dim\h=(\dim\g-\rank\g)/2
\]
      characterizes the Chevalley, or split, symmetric pairs.  The difference
      between the two sides of the corresponding inequality defines a nonnegative
      $\theta$-divergence defect.  Its vanishing is also the critical homogeneity
      condition for the leading regular nilpotent germ on the fixed side.
      Semisimple descendants remain split symmetric pairs, although they need not
      remain equal rank.
\item The principal-$\slie_2$ calculation gives the rank formula
      \[
       \rank d(\chi|_{\cQ})=\#\{i:m_i\text{ is odd}\}=\rank H
\]
      Thus the Jacobi family maps dominantly onto the fixed part of the
      adjoint quotient, with relative dimension $\rank G-\rank H$, and is
      generically \'etale when the ranks are equal.
\item A relation
      \[
       \theta(Q(q))=\Ad(a(q))Q(q)
\]
      turns the descent problem into a Cartan-map lifting problem for which
      norm conditions and Galois cohomology determine the rational branches.
      Square classes $\gamma_\alpha$ in certain root-$\SL_2$ factors occur here, 
      and a Kummer kernel indexes twisted square-root
      covers on which the Cartan lift is an $F$-rational morphism.
\item On the square-root chart $q_\alpha=z_\alpha^2$, the $\theta$-fixed Jacobi
      family is linear in the $z_\alpha$.  Splitting the root spaces by
      height parity gives a determinant $R(z)$ for which
      \[
       \abs{D_H^G(Y(z))}^{1/2}
       =C\abs{R(z)}\prod_\alpha\abs{z_\alpha}
\]
      while the Jacobian of the $H$-quotient map along the same family is
      $C'R(z)$.  Hence the relative quotient density for $H^\circ$ is
      $C\prod_\alpha d^\times z_\alpha$ for compatible measure normalizations;
      in fact $R(z)$ and the ambient
      Jacobi Jacobian are monomials.  A direct weight-space argument then makes
      the Cartan-lift Iwasawa heights, for each fixed orbital variable and after
      a finite clopen refinement, piecewise affine in $\val(z_\alpha)$.
\end{enumerate}

The first two points are geometric.  The rational descent in \textup{(iii)} is
formulated after choosing the $F$-split, $\theta$-split maximal torus used
below.  The logarithmic-density statement in \textup{(iv)}, and the
descent-chart theorem that packages it, require the equal-rank condition.
When the ranks differ, the Jacobi map has relative dimension
$\rank G-\rank H$, which agrees with the geometric dimension of the residual
centralizer torus $T_H\backslash T_G$ on the fixed side.  When that torus is
$F$-split, this number is also the rank of its valuation lattice.

The pair
\[
 (G,H)=(\Sp_{2n},\GL_n)
\]
usually called type CI, is treated as a worked example.  Hurwitz coordinates
make the relative density logarithmic, and the Cartan lift factors through
commuting long-root $\SL_2$ subgroups.  The example illustrates all parts of
the general construction.

The main results of the paper are local and stand independently of a complete
trace-formula expansion: the descent-chart theorem constructs the rational
families and their nilpotent limits, while the Jacobi determinant identity
and the weight-space argument identify the quotient densities and, pointwise
in the orbital variable, the Iwasawa-height behavior on those families.  Thus
the result here should be
viewed as geometric input for, rather than a proof of, the polynomial
fixed-side trace-formula expansion.  We do not carry out the second
fixed-side truncation: a follow-up paper will specify suitable small neighborhoods
in the invariant quotient, describe their pullbacks to the Jacobi charts,
justify the required changes in order of integration, and evaluate the
resulting integrals of the truncation function as weights on nilpotent orbital
terms.

Section~2 isolates the fixed-side local problem and recalls the rank-one
model.  Sections~3 and 4 characterize $\theta$-divergence and study
semisimple descendants.  Sections~5 and 6 construct rank-one Cartan lifts and
organize their rational descent.  Section~7 develops the Jacobi family,
identifies the $\theta$-fixed part of the adjoint quotient, and proves generic
\'etaleness in equal rank.  Section~8 proves the height-parity determinant
identity, the logarithmic quotient-density formula, and the descent-chart
theorem.  Section~9 studies the resulting piecewise-affine Iwasawa-height
coordinates and explains how they enter the second truncation over small
quotient neighborhoods.  Section~10 gives the illustrative symplectic
calculation.

\section*{Acknowledgements}

The mathematical ideas and many of the calculations underlying this paper were developed during 2008--2011, beginning with the work recorded in \cite{SparlingThesis} and continuing during the author's postdoctoral years.  The author is grateful to Tasho Kaletha, Robert Kottwitz, Zhengyu Mao, Yiannis Sakellaridis, and Diana Shelstad for helpful discussions and encouragement during those years; this paper is, among other things, a belated continuation of those conversations.

AI tools were used during the final stages of writing for assistance with exposition, proofreading, and formatting; the author reviewed the final text and takes responsibility for its contents.

\section{\texorpdfstring{The $\theta$-fixed local problem}{The theta-fixed local problem}}
\label{sec:weighted-origin}

\subsection{The two infinitesimal sides}
For context, the $\theta$-split expansion established in
\cite{SparlingRankOne,SparlingThetaSplit} has the formal shape
\begin{equation}\label{eq:theta-split-formal}
 J^-(f,\mu)
 =\sum_M\sum_{T\in\mathcal T_M}c(M,T)
   \int_{\tT_{\mathrm{reg}}}
   \abs{D^G(X)}^{1/2}
   \int_{(A_MZ_H(T))\backslash G}
       f(g^{-1}Xg)\,\omega_M(g,\mu)\,dg\,dX
\end{equation}
where $\tT=\Lie(T)\subset\p$.  This is the source of
\eqref{eq:theta-split-weight} and of its toric polynomial approximation.
These results provide an essentially classical version of the truncation framework.

The $\theta$-fixed side is
\begin{equation}\label{eq:theta-fixed-side}
 J^+(\widehat f,\phi)
 =\int_{\h}\int_{H\backslash G}
   \widehat f(g^{-1}Xg)\phi(g)\,dg\,dX
\end{equation}
For each fixed $g$, Weyl integration in the $X$-variable is valid, but
using it under the remaining $g$-integration raises uniformity and
absolute-convergence issues near the nilpotent fiber.  A second truncation
isolates a small neighborhood of that fiber in the invariant quotient.  The
singular contribution from this neighborhood is then reorganized by local
coordinates and a change in the order of integration.  Ordinary germ theory
still describes the local nilpotent singularity and its homogeneity, but the
weight appearing in the final nilpotent term is produced by this neighborhood
integration.

\subsection{The rank-one fixed-side calculation}
\cite{SparlingRankOne} provides the explicit trace formula for $F^\times\backslash\SL_2(F)$. Choose
\[
 \h_\epsilon=\left\{
 \begin{pmatrix}x&0\\0&-x\end{pmatrix}:|x|\leq|\epsilon|
 \right\}
\]
The complement of $\h_\epsilon$ contributes a regular integral.  On
$\h_\epsilon$, the Iwasawa decomposition and a change in the order of
integration convert the singular part into an integral over the regular
nilpotent line.  For $\epsilon$ sufficiently small relative to the test
function, the polynomial approximation is
\begin{equation}\label{eq:rank-one-fixed-term}
 (1-|\varpi|)\int_F
 \widetilde f
 \begin{pmatrix}0&2n\\0&0\end{pmatrix}
 \bigl(1+\val(n)+\mu-\val(\epsilon)\bigr)\,dn
\end{equation}
The (second) factor in parentheses comes from integrating the original truncation function
over the portion of the adjoint quotient represented by $\h_\epsilon$.  Thus the new polynomial is obtained by
an explicit local calculation.  The same rank-one formula also retains square-class
parameters on the semisimple side; those parameters foreshadow the norm and
cohomological conditions in Sections~5 and 6.

\subsection{The local integral in higher rank}
The local fixed-side problem is modeled by integrals of the
following form; this is also the form produced by the trace-formula expansion motivating the construction.  Let $\chi_H:\h\to\h/\!/H$ be the invariant quotient.  On a rational
branch of its regular locus (or after the finite cover needed to obtain such a
branch), let $S$ be a regular section over a dense open subset $U$.  With compatible
measures one encounters
\begin{equation}\label{eq:weighted-weyl}
 \int_U\abs{D^H(S(s))}^{1/2}
   \int_{H\backslash G}\int_{Z_H(S(s))\backslash H}
   \widehat f(g^{-1}h^{-1}S(s)hg)
   \overline\omega(g,\mu)\,dh\,dg\,ds
\end{equation}
Here
$\overline\omega$ is the original truncation function, not the already
expanded weight $\omega_M$ from the $\theta$-split side.  Also, $ds$ is the
quotient measure in invariant coordinates on $\h/\!/H$; passing from a Cartan
to invariant coordinates absorbs one square root of the ordinary Weyl
discriminant.  For a reductive group $L$ and a regular semisimple element
$X\in\Lie(L)$ we use
\[
 D^L(X)=\det\bigl(\ad(X);\Lie(L)/\Lie(L)_X\bigr)
\]

Since $\overline\omega$ is a function on $H\backslash G$, integration in
stages combines the two inner quotients in \eqref{eq:weighted-weyl}:
\begin{equation}\label{eq:fixed-combined-orbit}
 \int_{Z_H(X)\backslash H}\int_{H\backslash G}
   \widehat f(g^{-1}h^{-1}Xhg)\overline\omega(g,\mu)\,dg\,dh
 =\int_{Z_H(X)\backslash G}
   \widehat f(g^{-1}Xg)\overline\omega(g,\mu)\,dg
\end{equation}
Thus the fixed-side singularity is most naturally compared with an ambient
$G$-orbital integral.  One discrepancy is the difference between the
regular centralizers in $H$ and in $G$.

\subsection{The residual centralizer torus}
Let $X\in\h(F)$ be regular semisimple in both $\h$ and $\g$, and put
\[
 T_G=Z_G(X)^\circ\qquad
 T_H=Z_{H^\circ}(X)^\circ=(T_G\cap H^\circ)^\circ
\]
Finite component groups introduce no additional continuous variables, but can
alter finite orbit indexing and quotient normalizations.  We suppress these
effects in the continuous calculation and restore the relevant finite-cover
qualifications below.  The tori $T_G$ and $T_H$ have dimensions
$\rank G$ and $\rank H$, respectively, and
\begin{equation}\label{eq:residual-torus-dimension}
 \dim(T_H\backslash T_G)=\rank G-\rank H
\end{equation}
With compatible quotient measures, integration in stages gives
\begin{equation}\label{eq:centralizer-disintegration}
 \int_{T_H(F)\backslash G(F)}\Psi(g)\,dg
 =\int_{T_G(F)\backslash G(F)}
   \int_{T_H(F)\backslash T_G(F)}\Psi(tg)\,dt\,dg
\end{equation}
For
\[
 \Psi(g)=\widehat f(g^{-1}Xg)\,\overline\omega(g,\mu)
\]
the orbital factor is unchanged by $t\in T_G(F)$.  Hence the inner integral
is the residual weight
\begin{equation}\label{eq:residual-arthur-weight}
 \omega_X^{\mathrm{cen}}(g,\mu)
 =\int_{T_H(F)\backslash T_G(F)}
   \overline\omega(tg,\mu)\,dt
\end{equation}
Define the ambient-normalized weighted orbital integral
\[
 J_X^G(\widehat f,\omega_X^{\mathrm{cen}})
 =\abs{D^G(X)}^{1/2}
   \int_{T_G(F)\backslash G(F)}
      \widehat f(g^{-1}Xg)\omega_X^{\mathrm{cen}}(g,\mu)\,dg
\]
Then \eqref{eq:weighted-weyl} takes the normalized form
\begin{equation}\label{eq:weighted-weyl-normalized}
 \int_U
 \frac{ds}{\abs{D_H^G(S(s))}^{1/2}}
 J_{S(s)}^G(\widehat f,\omega_{S(s)}^{\mathrm{cen}})
 \qquad
 D_H^G(X)=\frac{D^G(X)}{D^H(X)}
\end{equation}
This is the source of the quotient density
$ds/\abs{D_H^G}^{1/2}$ studied below.

\begin{remark}\label{rem:equal-rank-reduction} (Reduction to the equal rank case)
Let $C_X=T_H\backslash T_G$.  Its unbounded directions are measured by
the valuation homomorphism
\[
 \nu_{C_X}:C_X(F)\longrightarrow
 \Hom\bigl(X_F^*(C_X),\mathbb Z\bigr)
 \qquad
 \langle\chi,\nu_{C_X}(c)\rangle=-\val(\chi(c))
\]
Here $X_F^*(C_X)$ denotes the group of $F$-rational characters of $C_X$.
The kernel of $\nu_{C_X}$ is compact and its image has finite index; the
target has rank $\rank_F C_X$.  Moreover, the image of $T_G(F)$ in $C_X(F)$ has finite index,
since the cokernel injects into the finite pointed set $H^1(F,T_H)$.  Thus,
up to finitely many cosets, \eqref{eq:residual-arthur-weight} is an integral
over a lattice of rank $\rank_F C_X$ \cite{Serre}.  After fixing a suitable truncation
chamber, this is the usual source of Arthur-like orbital integral weights.  When $C_X$ is $F$-split, this rank is
$\rank G-\rank H$.  This suggests separating an Arthur-like integration along
$T_H\backslash T_G$ from the equal-rank nilpotent Jacobi analysis studied here;
no canonical factorization of the symmetric pair is asserted.
\end{remark}

When $\rank H=\rank G$, the connected tori $T_H\subset T_G$ have the same
dimension and hence are equal.  The continuous residual integration therefore
disappears, and \eqref{eq:weighted-weyl-normalized} is directly a weighted
ambient $G$-orbital integral.

In this equal-rank case the continuous residual-centralizer integration has
disappeared, so the remaining singular issue is concentrated in the quotient
variables.  A second truncation procedure cuts out a sufficiently small neighborhood
of the nilpotent point in $\h/\!/H$.  On a rational branch we will study a family
$u(s)$ such that
\[
 Q(s):=\Ad(u(s)^{-1})S(s)
\]
extends to a regular nilpotent element of $\g$.  Pulling the excised
neighborhood to such a branch and conjugating by $u(s)$ transfers the
singularity to this single controlled degeneration, while the original
truncation function is shifted by $u(s)$.  The intended analytic step is then
to change the order of integration: integration in the quotient variables
produces the weight multiplying the resulting nilpotent orbital term.  The
calculation in this paper identifies the rational branches of $u$, the
quotient coordinates and density needed for that integration, and the
valuation variables on which the shifted truncation function depends.

\subsection{Motivation and geometric requirements}
The equality of polynomial approximations in the rank-one trace formula was
the motivation for expecting a polynomial answer in higher rank.  The
calculation will establish the following concrete properties.
\begin{enumerate}[label=\textup{(\alph*)},leftmargin=2.3em]
\item a sparse regular family $Q(q)$ extends to a principal nilpotent element;
\item $\tau(u(q))$ lies in a fixed $\theta$-split torus and
      $\Ad(u(q))Q(q)\in\h$;
\item the square classes and cohomological branches of $u(q)$ are explicit;
\item on the square-root Jacobi chart the $H^\circ$-relative quotient density
      is a constant multiple of product multiplicative Haar measure;
\item for each fixed orbital variable, after a finite clopen refinement, the
      Iwasawa-height functions entering the fixed-side truncation are affine on
      polyhedral chambers in the valuations of the square-root Jacobi
      coordinates.
\end{enumerate}
The first three conditions control the nilpotent degeneration and its
rationality.  The last two put the excised quotient neighborhood and the
shifted truncation function into coordinates suited to the integrations which
produce the polynomial weight.

\section{\texorpdfstring{$\theta$-divergence and split symmetric pairs}{theta-divergence and split symmetric pairs}}

\subsection{Basic notation}
Let $F$ be a field of characteristic zero.  Let $G$ be a
connected reductive group over $F$, let $\theta:G\to G$ be an involution, and
let $H=G^\theta$ be the fixed-point group.  Then $H$ is smooth; all Lie-algebra
statements below depend only on its identity component.  Unless a subscript
$F$ is displayed, ranks are absolute ranks (equivalently, ranks after extension
to $\Fb$).  Write
\[
 \g=\Lie(G)\qquad \h=\Lie(H)\qquad
 \g=\h\oplus\p
\]
where $\h$ and $\p$ are the $+1$ and $-1$ eigenspaces of $d\theta$.
The bracket satisfies
\[
 [\h,\h]\subset\h,\qquad [\h,\p]\subset\p
 \qquad [\p,\p]\subset\h
\]
For $X\in\g$ we put
\[
 \g_X=Z_\g(X)\qquad \h_X=\g_X\cap\h
 \qquad \p_X=\g_X\cap\p
\]

Choose a nondegenerate invariant symmetric bilinear form
$B$ on $\g$ that is $\theta$-invariant.  On the derived algebra take the
Killing form.  On the center, decompose into the $+1$ and $-1$ eigenspaces of
$\theta$, choose nondegenerate symmetric forms separately on the two
eigenspaces, and declare them orthogonal.  Combining these forms gives the
required $B$.  The subspaces $\h$ and $\p$ are orthogonal for $B$.

\subsection{The centralizer identity}
The following observation is the source of all dimension formulas in this
paper.

\begin{proposition}\label{prop:rank-identity}
For every $X\in\p$,
\begin{equation}\label{eq:rank-identity}
 \dim\h-\dim\h_X=\dim\p-\dim\p_X
\end{equation}
Equivalently,
\begin{equation}\label{eq:centralizer-difference}
 \dim\h-\dim\p=\dim\h_X-\dim\p_X
\end{equation}
\end{proposition}

\begin{proof}
Because $X\in\p$, the map $\ad X$ sends $\h$ to $\p$ and $\p$ to $\h$.
For $Y\in\h$ and $Z\in\p$, invariance of $B$ gives
\[
 B([X,Y],Z)=-B(Y,[X,Z])
\]
Thus the maps
\[
 \ad X:\h\longrightarrow\p
 \qquad
 \ad X:\p\longrightarrow\h
\]
are transposes up to sign.  They have the same rank.  Their kernels are
$\h_X$ and $\p_X$, respectively, which proves \eqref{eq:rank-identity}.
\end{proof}

\begin{remark}
For semisimple $X$, one may instead use the induced isomorphism on
$\g/\g_X$.  The bilinear-form proof shows that semisimplicity is unnecessary.
\end{remark}

A \emph{Cartan subspace} is a maximal abelian subspace of $\p$ consisting of
semisimple elements.  Fix a Cartan subspace $\aA\subset\p$, and write
\[
 \qquad r_\theta=\dim\aA
 \qquad \mM=Z_\h(\aA)
\]
Call $X\in\aA$ \emph{relatively regular} if its centralizer in $\p$ has
minimal dimension, equivalently if $X$ lies off all restricted-root
hyperplanes.  For such an $X$, one has
\[
 \p_X=\aA
 \qquad \h_X=\mM
\]
Applying \cref{prop:rank-identity} gives the exact defect formula.

\begin{theorem}\label{thm:defect-formula}
Let $r=\rank\g$.  Then
\begin{equation}\label{eq:defect-formula}
 \dim\h-\frac{\dim\g-r}{2}
 =\frac{(r-r_\theta)+\dim\mM}{2}
\end{equation}
In particular,
\begin{equation}\label{eq:dimension-lower-bound}
 \dim\h\geq\frac{\dim\g-r}{2}
\end{equation}
\end{theorem}

\begin{proof}
For relatively regular $X\in\aA$, \cref{prop:rank-identity} gives
\[
 \dim\h-\dim\p=\dim\mM-r_\theta
\]
Since $\dim\g=\dim\h+\dim\p$, we obtain
\[
 2\dim\h=\dim\g+\dim\mM-r_\theta
\]
Subtracting $\dim\g-r$ and dividing by $2$ gives
\eqref{eq:defect-formula}.  Both terms on the right are nonnegative.
\end{proof}

We call the nonnegative quantity
\begin{equation}\label{eq:theta-divergence-defect}
 \delta_\theta(G,H)
 :=\dim\h-\frac{\dim\g-\rank\g}{2}
\end{equation}
the \emph{$\theta$-divergence defect}.

\subsection{The extremal case}

\begin{definition}\label{def:split-symmetric}
Assume first that $G$ is semisimple.  We call $(G,H,\theta)$ a
\emph{split symmetric pair} if its $\theta$-divergence defect vanishes,
equivalently if
\begin{equation}\label{eq:extremal-case}
 \dim\h=\frac{\dim\g-\rank\g}{2}
\end{equation}
For reductive $G$, we apply the definition to the derived group.
\end{definition}

Under this convention, splitness of the derived pair does not by itself imply
that $\delta_\theta(G,H)=0$ for the full reductive Lie algebra.  Indeed, when
the derived pair is split,
\[
 \delta_\theta(G,H)=\dim\mathfrak z(\g)^\theta
\]
Thus the full extremal identity additionally requires $\theta$ to act by
$-1$ on the center; any $\theta$-fixed central directions factor off as
ordinary additive quotient coordinates.

\begin{theorem}\label{thm:split-symmetric-equivalences}
Assume that $G$ is semisimple.  The following are equivalent after extension
to an algebraic closure of $F$.
\begin{enumerate}[label=\textup{(\alph*)},leftmargin=2.3em]
\item $(G,H,\theta)$ is a split symmetric pair.
\item $r_\theta=\rank\g$.
\item $\p$ contains a Cartan subalgebra of $\g$.
\item There is a Cartan subalgebra $\tT\subset\g$ on which $\theta$ acts by
      $-1$.
\item The involution $\theta$ is conjugate to a Chevalley involution.
\end{enumerate}
When these conditions hold, $Z_\h(\aA)=0$ for every Cartan subspace
$\aA\subset\p$.
\end{theorem}

\begin{proof}
By \cref{thm:defect-formula}, condition (a) is equivalent to
$r_\theta=\rank\g$ and $\mM=0$.  If $r_\theta=\rank\g$, then a Cartan subspace
$\aA$ has the dimension of a Cartan subalgebra and is therefore itself a
Cartan subalgebra.  Its centralizer in $\g$ is $\aA$, so its centralizer in
$\h$ is zero.  This proves the equivalence of (a)--(c), and (c) is visibly
equivalent to (d).

If $\theta$ acts as $-1$ on a Cartan subalgebra $\tT$, then it sends each root
space $\g_\alpha$ to $\g_{-\alpha}$.  After rescaling a Chevalley basis, this
is the defining action of a Chevalley involution.  Conversely, a Chevalley
involution acts as $-1$ on the chosen Cartan subalgebra.  This proves the final
equivalence.  See also \cite{KostantRallis,HelminckWang}.
\end{proof}

Thus \cref{thm:split-symmetric-equivalences} gives a numerical
characterization of the maximal-rank, or Chevalley-involution, condition over
an algebraic closure: the extremal dimension condition is exactly the
condition that $\p$ contain a Cartan subalgebra of $\g$.  The additional
rational requirement that the chosen maximal $\theta$-split torus be
$F$-split is imposed separately when needed.

\begin{remark}
The defect $\delta_\theta(G,H)$ remains meaningful outside the split case.
In the trace-formula application, its vanishing is the critical homogeneous
degree at which the fixed-side semisimple expansion acquires the singular
behavior studied below; the positive defect measures the distance from that
threshold.
\end{remark}

\begin{remark}\label{rem:homogeneity-divergence}
Assume in this remark that $G$ is semisimple.  The same equality is visible in the homogeneity of the leading regular
nilpotent germ on the $\theta$-fixed side.  Before the Weyl density is
inserted, the leading unnormalized germ has degree
$-(\dim\g-\rank\g)/2$ under $X\mapsto aX$; see
\cite[\S17.7]{KottwitzClay}.  If $s$ denotes homogeneous coordinates on
$\h/\!/H^\circ$, then $ds$ has degree $(\dim\h+\rank\h)/2$: the sum of the
degrees of a set of basic invariants is $\rank\h+\abs{\Phi_H^+}$.  Passing to
the finite quotient by $H/H^\circ$ does not change this generic radial degree.  Meanwhile
$\abs{D^H}^{1/2}$ has degree $(\dim\h-\rank\h)/2$.  Thus the Weyl density
$\abs{D^H}^{1/2}ds$ has degree $\dim\h$, and the resulting radial degree is
\[
 \dim\h-\frac{\dim\g-\rank\g}{2}
\]
The split case is therefore the critical homogeneous degree at which
logarithmic radial behavior can occur; obtaining an actual logarithmic term
also requires a nonzero leading germ and a compatible choice of neighborhood.
This is one analytic motivation
for \cref{def:split-symmetric,cor:split-centralizers}.
\end{remark}

\subsection{Classification over an algebraic closure}
Over an algebraic closure, a Chevalley involution is unique up to conjugacy
on each simple factor.  Its fixed root datum is obtained directly by pairing
opposite root spaces and taking the $+1$ eigenspaces.  

Involutions of complex simple Lie algebras and their fixed-point
subalgebras are tabulated in \cite[Table~7]{OnishchikVinberg}; see also
\cite{HelminckClassification} for the formulation in terms of semisimple
symmetric pairs and restricted root systems.  Restricting that classification
to the Chevalley involution gives \cref{tab:classification}.

\begin{table}[ht]
\centering
\renewcommand{\arraystretch}{1.12}
\begin{tabular}{@{}cclc@{}}
\toprule
Type of $\g$ & $\h=\g^\theta$ & $\rank\h$ & Equal rank?\\
\midrule
$A_{n-1}$ & $\solie_n$ & $\lfloor n/2\rfloor$ & only $A_1$\\
$B_n$ & $\solie_n\oplus\solie_{n+1}$ & $n$ & yes\\
$C_n$ & $\glie_n$ & $n$ & yes\\
$D_n$ & $\solie_n\oplus\solie_n$ & $2\lfloor n/2\rfloor$ & $n$ even\\
$E_6$ & $\splie_8$ & $4$ & no\\
$E_7$ & $\slie_8$ & $7$ & yes\\
$E_8$ & $\solie_{16}$ & $8$ & yes\\
$F_4$ & $\splie_6\oplus\slie_2$ & $4$ & yes\\
$G_2$ & $\slie_2\oplus\slie_2$ & $2$ & yes\\
\bottomrule
\end{tabular}
\caption{Split symmetric pairs and the equal-rank subfamily.}
\label{tab:classification}
\end{table}

For a semisimple algebra, $\theta$ preserves each simple factor in the
split case.  Indeed, if it exchanged two isomorphic factors, the
diagonal fixed subalgebra would violate \eqref{eq:extremal-case} on
that pair of factors.  Hence the semisimple classification is obtained by
taking products of the entries in \cref{tab:classification}.

The equal-rank column will be important for the explicit Jacobi construction.
For an irreducible root system it is equivalent to
\[
 -1\in W
\]
The excluded simple types are
\[
 A_n\;(n\geq2),\qquad D_{2m+1},\qquad E_6
\]

\section{Centralizers and semisimple descent}

Throughout this section $G$ is semisimple.  For a reductive group, the
statements apply to the derived symmetric pair, as in
\cref{def:split-symmetric}.  The centralizer identity becomes especially
rigid for a split symmetric pair.  Let $r=\rank\g$.

\begin{corollary}\label{cor:split-centralizers}
Assume that $(G,H,\theta)$ is a split symmetric pair.  Then for every $X\in\p$,
\begin{equation}\label{eq:split-centralizer-difference}
 \dim\p_X-\dim\h_X=r
\end{equation}
\end{corollary}

\begin{proof}
The defining equality gives $\dim\p-\dim\h=r$.  Apply
\cref{prop:rank-identity}.
\end{proof}

\subsection{Descendants}
If $X\in\p$ is semisimple, the identity component $G_X^\circ$ of its
centralizer is reductive and stable under $\theta$.  Its Lie algebra has the
eigenspace decomposition
\[
 \g_X=\h_X\oplus\p_X
\]
We call the connected symmetric pair determined by
$(G_X^\circ,\theta|_{G_X^\circ})$ the semisimple descendant at $X$; finite
component groups may be restored separately when orbital integrals are
formed.

\begin{theorem}\label{thm:descendants-split}
Every semisimple descendant of a split symmetric pair is again a split symmetric pair.
\end{theorem}

\begin{proof}
Choose a Cartan subspace $\aA_X\subset\p$ containing $X$.  Since the ambient
pair is split, \cref{thm:split-symmetric-equivalences} makes
$\aA_X$ a Cartan subalgebra of $\g$ on which $\theta$ acts by $-1$.  In
particular the semisimple centralizer contains a Cartan subalgebra of $\g$, so
\[
 \rank\g_X=\rank\g=r
\]
By \cref{cor:split-centralizers},
\[
 \dim\p_X-\dim\h_X=r
\]
Since $\dim\g_X=\dim\h_X+\dim\p_X$, it follows that
\[
 2\dim\h_X=\dim\g_X-r
 =\dim\g_X-\rank\g_X
\]

It remains only to match this with the reductive convention in
\cref{def:split-symmetric}.  The connected center of $G_X^\circ$ is contained
in the Cartan with Lie algebra $\aA_X$, so $\theta$ acts on it by inversion;
its Lie algebra therefore contributes only to $\p_X$.  Removing this $\theta$-split central torus from the preceding equality gives
\[
 2\dim [\g_X,\g_X]^\theta
 =\dim [\g_X,\g_X]-\rank [\g_X,\g_X]
\]
which is precisely the vanishing of the $\theta$-divergence defect for the
derived symmetric pair.
\end{proof}

This closure property is a key reason the class is suitable for a recursive
fixed-side analysis near nonregular strata.

\subsection{Root-subsystem formula}
Fix a Cartan subalgebra $\aA\subset\p$, and let
$\Phi=\Phi(\g,\aA)$ be the corresponding root system.  Since
$\theta$ acts as $-1$ on $\aA$, it exchanges $\g_\alpha$ and $\g_{-\alpha}$.
For $X\in\aA$, set
\[
 \Phi_X=\{\alpha\in\Phi:\alpha(X)=0\}
\]
Then
\[
 \g_X=\aA\oplus\bigoplus_{\alpha\in\Phi_X}\g_\alpha
\]
Choose a positive system $\Phi_X^+$.

\begin{proposition}\label{prop:root-centralizer-dimensions}
For $X\in\aA$,
\[
 \dim\h_X=\abs{\Phi_X^+}
 \qquad
 \dim\p_X=r+\abs{\Phi_X^+}
\]
The descendant root system is $\Phi_X$, and $\theta|_{\g_X}$ is its Chevalley
involution.
\end{proposition}

\begin{proof}
For each pair $\{\alpha,-\alpha\}$ in $\Phi_X$, the direct sum
$\g_\alpha\oplus\g_{-\alpha}$ has one dimension in $\h$ and one in $\p$.
The Cartan $\aA$ lies entirely in $\p$.  Summing over positive roots gives the
formula.  The final assertion follows because $\theta$ acts by $-1$ on
$\aA$ and exchanges opposite root spaces.
\end{proof}

Thus the semisimple centralizer strata in $\mathfrak{a}$ are indexed by the
vanishing-root subsystems $\Phi_X$, and the same split symmetric-space
geometry recurs on each stratum.

\subsection{Regular nilpotent centralizers}
After extension to $\Fb$, Kostant--Rallis theory implies that a
split symmetric pair contains regular nilpotent elements in $\p$
\cite{KostantRallis}.  Their centralizers satisfy the extremal boundary case of
\cref{cor:split-centralizers}.

\begin{corollary}\label{cor:regular-nilpotent-centralizer}
Let $e\in\p$ be regular nilpotent in $\g$.  Then
\[
 \h_e=0
 \qquad
 \g_e=\p_e
 \qquad
 \dim\p_e=r
\]
\end{corollary}

\begin{proof}
Regularity gives $\dim\g_e=r$.  On the other hand,
\[
 \dim\p_e-\dim\h_e=r
\]
by \cref{cor:split-centralizers}.  Adding and subtracting this equation from
$\dim\p_e+\dim\h_e=r$ gives the result.
\end{proof}

This records the limiting centralizer behavior inside the symmetric nilpotent
cone.  The endpoint of the family of Jacobi elements used below is a regular nilpotent element of $\g$
and need not itself lie in $\p$.

\section{Rank-one sections of the Cartan map}

\subsection{\texorpdfstring{The $\gamma$-twisted $\SL_2$ calculation}{The gamma-twisted SL2 calculation}}
Let $A_1\subset\SL_2$ be the diagonal torus.  Every involution that acts by
inversion on $A_1$ has, in a suitable $F$-basis, the form
\begin{equation}\label{eq:sl2-involution-gamma}
 \theta_\gamma:
 \begin{pmatrix}a&b\\c&d\end{pmatrix}
 \longmapsto
 \begin{pmatrix}d&c/\gamma\\ \gamma b&a\end{pmatrix}
 \qquad \gamma\in F^\times
\end{equation}
It is conjugation by
$\left(\begin{smallmatrix}0&1\\ \gamma&0\end{smallmatrix}\right)$.
The square class of $\gamma$ is intrinsic: rescaling a compatible pair of
root vectors changes $\gamma$ by a square.  The fixed torus is $F$-split
exactly when $\gamma$ is a square.

Let
\[
 E_\gamma=F[T]/(T^2-\gamma)
\]
be the associated quadratic \emph{\'etale} algebra; it is a field when
$\gamma$ is nonsquare and is isomorphic to $F\times F$ when $\gamma$ is a
square.  For $r,s\in F$ and $z\in F^\times$ with
\begin{equation}\label{eq:rank-one-norm}
 r^2-\gamma s^2=z
\end{equation}
define
\begin{equation}\label{eq:gamma-cayley}
 c_\gamma(r,s;z)=
 \begin{pmatrix}
 r&s/z\\
 \gamma s&r/z
 \end{pmatrix}
\end{equation}

\begin{lemma}[The rank-one norm section]\label{lem:sl2-section}
The matrix $c_\gamma(r,s;z)$ belongs to $\SL_2(F)$ and satisfies
\begin{equation}\label{eq:sl2-tau}
 \theta_\gamma(c_\gamma(r,s;z))^{-1}c_\gamma(r,s;z)
 =\begin{pmatrix}z&0\\0&z^{-1}\end{pmatrix}
\end{equation}
Consequently, a point $\diag(z,z^{-1})\in A_1(F)$ has a lift through the
rank-one Cartan map if and only if
\[
 z\in N_{E_\gamma/F}(E_\gamma^\times)
\]
\end{lemma}

\begin{proof}
Equation \eqref{eq:rank-one-norm} gives
\[
 \det c_\gamma(r,s;z)=\frac{r^2-\gamma s^2}{z}=1
\]
A direct multiplication using \eqref{eq:sl2-involution-gamma} gives
\eqref{eq:sl2-tau}.  Conversely, suppose
$u=\left(\begin{smallmatrix}x&y\\w&t\end{smallmatrix}\right)$ satisfies
$\theta_\gamma(u)^{-1}u=\diag(z,z^{-1})$.  Equivalently,
$\theta_\gamma(u)=u\diag(z^{-1},z)$, so
\[
 t=x/z,\qquad w=\gamma yz
\]
Putting $r=x$ and $s=yz$, the equation $\det u=1$ becomes
$r^2-\gamma s^2=z$.  Since
$r^2-\gamma s^2=N_{E_\gamma/F}(r+sT)$, the norm condition is
both necessary and sufficient.
\end{proof}

If $\gamma=\delta^2$ is a square, take
\[
 r=\frac{1+z}{2},\qquad
 s=\frac{1-z}{2\delta}
\]
When $\delta=1$, this gives the familiar rational Cayley section
\begin{equation}\label{eq:cayley-section}
 c(z)=
 \begin{pmatrix}
 \dfrac{1+z}{2}&\dfrac{z^{-1}-1}{2}\\[5pt]
 \dfrac{1-z}{2}&\dfrac{z^{-1}+1}{2}
 \end{pmatrix}
\end{equation}
Setting $\gamma=1$ is legitimate for a split rank-one
fixed torus, but not for a nonsplit rational form.  In the latter case the
norm condition is part of the lifting problem and naturally becomes part of the cohomological parametrization below.

\subsection{\texorpdfstring{Root $\SL_2$ subgroups and strongly orthogonal products}{Root SL2 subgroups and strongly orthogonal products}}
The preceding norm-section calculation will now be expressed in root-theoretic
terms.  This reformulation separates the rank-one construction from the
particular matrix realization and makes it available in higher rank.

Let $(G,H,\theta)$ again be a split symmetric pair with $G$ semisimple.  For the rational
root construction we fix an $F$-split maximal torus $A\subset G$ on which
$\theta(a)=a^{-1}$; equivalently, $\Lie(A)\subset\p$.  All statements in
Sections~3 and 4 were fundamentally geometric and did not require this additional rational
choice.  For a root $\beta$, choose an
$F$-rational root homomorphism
\[
 \iota_\beta:\SL_2\longrightarrow G
\]
compatible with a Chevalley basis.  There is a square class
$\gamma_\beta\in F^\times/(F^\times)^2$ such that
\[
 \theta\circ\iota_\beta
 =\iota_\beta\circ\theta_{\gamma_\beta}
\]
Choose a representative of this square class $\gamma_\beta$ and set
\[
 E_\beta=F[T]/(T^2-\gamma_\beta)
\]
For $w=r+sT\in E_\beta^\times$, put
$z=N_{E_\beta/F}(w)$ and define
\[
 u_\beta(w)=
 \iota_\beta\bigl(c_{\gamma_\beta}(r,s;z)\bigr)
\]
Then
\begin{equation}\label{eq:root-section-tau}
 \tau(u_\beta(w))=\beta^\vee(z)
\end{equation}
Thus the root-$\SL_2$ construction gives an explicit lift on the image of
the norm map $N_{E_\beta/F}$; when $\gamma_\beta$ is a square, it gives a
rational section on all of $\Gm$.

Let $\cB=\{\beta_1,\ldots,\beta_s\}$ be pairwise strongly orthogonal.
The corresponding root-$\SL_2$ subgroups commute.  Set
\[
 \mathcal R_\cB=\prod_{i=1}^s
       \operatorname{Res}_{E_{\beta_i}/F}\Gm
 \qquad
 \nu_\cB(w_1,\ldots,w_s)
 =\prod_{i=1}^s\beta_i^\vee
       \bigl(N_{E_{\beta_i}/F}(w_i)\bigr)
\]

\begin{theorem}\label{thm:orthogonal-section}
The map
\[
 \widetilde u_\cB(w_1,\ldots,w_s)
 =\prod_{i=1}^s u_{\beta_i}(w_i)
\]
is independent of the order of the factors and satisfies
\[
 \tau(\widetilde u_\cB(w))=\nu_\cB(w)
\]
If all $\gamma_{\beta_i}$ are squares, choosing their square roots reduces
this to a rational section on the split torus $\Gm^s$.
\end{theorem}

\begin{proof}
Strong orthogonality implies that the root subgroups associated with
$\pm\beta_i$ commute with those associated with $\pm\beta_j$.  The factors
and their $\theta$-images therefore commute, so the Cartan map factors
componentwise.  The assertion follows from
\eqref{eq:root-section-tau}.
\end{proof}

\begin{warning}\label{warn:not-descend}
The strongly orthogonal product is an explicit partial construction, not the
general construction of rational Cartan-lift branches used later.  Neither
the norm maps nor the isogeny generated by the strongly orthogonal coroots
should be inverted carelessly on $F$-points.  Even for a full-rank strongly
orthogonal system, the product of rank-one sections directly reaches only the
image of the relevant norm maps followed by the coroot isogeny; this image
need not be all of $A(F)$.  The square classes $\gamma_\beta$ measure the
norm obstruction for this explicit rank-one construction, while the finite
kernel of the coroot isogeny gives a separate obstruction to the strongly
orthogonal construction.  Neither obstruction by itself characterizes the
image of the full Cartan map.  Moreover, choosing rank-one lifts pointwise
does not by itself produce a rational algebraic branch.  Section~6 therefore
constructs the general branches instead from twisted square-root covers and
the Kummer kernel.  The strongly orthogonal construction remains useful as an
explicit realization when its norm and isogeny obstructions vanish, and it
reappears concretely in the symplectic example of Section~10.
\end{warning}

\subsection{Full-rank systems}
A strongly orthogonal system of cardinality $\rank G$ exists exactly in the
equal-rank cases of \cref{tab:classification}; equivalently, $-1\in W$.
This follows from the classification of Weyl-group involutions and Kostant's
cascade of strongly orthogonal roots; see \cite{Carter,KostantCascade}.
In the classical cases one may choose:
\begin{align*}
 C_n:&\quad 2e_1,\ldots,2e_n\\
 B_{2m}:&\quad e_1-e_2,e_1+e_2,\ldots,
              e_{2m-1}-e_{2m},e_{2m-1}+e_{2m}\\
 B_{2m+1}:&\quad e_1-e_2,e_1+e_2,\ldots,
              e_{2m-1}-e_{2m},e_{2m-1}+e_{2m},e_{2m+1}\\
 D_{2m}:&\quad e_1-e_2,e_1+e_2,\ldots,
              e_{2m-1}-e_{2m},e_{2m-1}+e_{2m}
\end{align*}
For such a system, the coroots generate $A$ up to isogeny.  Over an
algebraic closure, and over $F$ whenever the relevant
$\gamma_{\beta_i}$ are squares, \cref{thm:orthogonal-section} gives an
explicit product of rank-one lifts.  In general the norm images and the
finite coroot isogeny must both be descended on $F$-points.  This is the
cohomological issue considered next.

\section{Cohomological descent}

\subsection{The common descent cocycle}
Let $A$ be a maximal $\theta$-split and $F$-split torus and put
\[
 K=H\cap A
\]
The fibers of the multiplication map $H\times A\to HA$ are precisely the
$K$-orbits for
\[
 k\cdot(h,a)=(hk,k^{-1}a)
\]
Both the rational-point decomposition of $HA$ and the lifting problem for the
Cartan map are therefore governed by the same elementary descent calculation.

\begin{lemma}[Descent in $HA$]\label{lem:HA-descent}
For $h\in H(\Fb)$ and $a\in A(\Fb)$, the product $ha$ is $F$-rational if and
only if, for every $\sigma\in\operatorname{Gal}(\Fb/F)$,
\begin{equation}\label{eq:HA-descent-cocycle}
 h^{-1}\sigma(h)=a\sigma(a)^{-1}\in K(\Fb)
\end{equation}
The common value is a $K$-valued cocycle.  Changing the factorization of
$ha$ changes this cocycle by a $K$-coboundary.
\end{lemma}

\begin{proof}
The equality $\sigma(ha)=ha$ is equivalent to
\eqref{eq:HA-descent-cocycle}.  Its two sides lie in $H(\Fb)$ and $A(\Fb)$,
respectively, and hence in their intersection $K(\Fb)$.  The remaining
assertions are immediate from the usual change-of-factorization calculation.
\end{proof}

As a first consequence, the rational points of $HA$ need not be equal to
$H(F)A(F)$; the defect is exactly the cohomology class in the preceding
lemma.

\begin{proposition}\label{prop:double-cosets}
There is a natural bijection
\begin{equation}\label{eq:double-cosets}
 H(F)\backslash (HA)(F)/A(F)
 \simeq
 \ker\left(
 H^1(F,K)\longrightarrow H^1(F,H)\times H^1(F,A)
 \right)
\end{equation}
If $A$ is $F$-split, then $H^1(F,A)=1$ and this becomes
\begin{equation}\label{eq:double-cosets-split}
 H(F)\backslash (HA)(F)/A(F)
 \simeq
 \ker\left(H^1(F,K)\longrightarrow H^1(F,H)\right)
\end{equation}
\end{proposition}

\begin{proof}
Apply \cref{lem:HA-descent} to a factorization $x=ha$ over $\Fb$.  The
resulting $K$-cocycle is trivial after extension to both $H$ and $A$, and its
class is unchanged by left multiplication by $H(F)$ or right multiplication
by $A(F)$.  Conversely, a $K$-cocycle trivial in both $H$ and $A$ admits
trivializations $h$ and $a$ satisfying \eqref{eq:HA-descent-cocycle}, so the
product $ha$ descends to $(HA)(F)$.  The final assertion is Hilbert 90 for the
split torus $A$.
\end{proof}

\subsection{The Kummer obstruction for the Cartan map}
On a $\theta$-split torus one has $\theta(a)=a^{-1}$, and therefore
\[
 K=H\cap A=A[2],
 \qquad
 \tau(a)=\theta(a)^{-1}a=a^2
\]
Thus the same descent cocycle has a particularly simple interpretation for a
Cartan lift.  Let
\[
 \delta:A(F)/A(F)^2\longrightarrow H^1(F,A[2])
\]
be the Kummer map: if $t^2=a_0$, then $\delta(a_0)$ is represented by
$\sigma\mapsto t^{-1}\sigma(t)$.  We use the usual pointed-set conventions
for nonabelian Galois cohomology \cite{Serre}.

\begin{theorem}\label{thm:kummer-criterion}
Assume that $A$ is $F$-split.  For $a_0\in A(F)$, the following are equivalent.
\begin{enumerate}[label=\textup{(\alph*)},leftmargin=2.3em]
\item There exists $u\in G(F)$ with $\tau(u)=a_0$.
\item The Kummer class $\delta(a_0)$ maps to the trivial class in $H^1(F,H)$.
\end{enumerate}
Consequently, under the Kummer identification
$A(F)/A(F)^2\simeq H^1(F,A[2])$, the square-classes represented by values of
$\tau$ form the pointed subset
\begin{equation}\label{eq:tau-image-kernel}
 \delta\bigl(\{a_0\in A(F):a_0=\tau(u)
                   \text{ for some }u\in G(F)\}\bigr)
 =
 \ker\left(H^1(F,A[2])\longrightarrow H^1(F,H)\right)
\end{equation}
\end{theorem}

\begin{proof}
Choose $t\in A(\Fb)$ with $t^2=a_0$.  The Kummer cocycle
\[
 c_\sigma=t^{-1}\sigma(t)\in A[2]
\]
also equals $t\sigma(t)^{-1}$, since every element of $A[2]$ is its own
inverse.  Hence \cref{lem:HA-descent}, with $a=t$, says that a trivialization
\[
 c_\sigma=h^{-1}\sigma(h),\qquad h\in H(\Fb)
\]
is exactly what is needed for $u=ht$ to be $F$-rational.  For such a product,
\[
 \tau(u)=\tau(ht)=t^2=a_0
\]
Conversely, if $u\in G(F)$ satisfies $\tau(u)=a_0=\tau(t)$, then
$h=ut^{-1}$ is fixed by $\theta$, hence lies in $H(\Fb)$; applying
\cref{lem:HA-descent} to the rational product $u=ht$ gives
$c_\sigma=h^{-1}\sigma(h)$.  This proves the equivalence and the final
identification of square classes.
\end{proof}

\subsection{Twisted square-root covers}
The Kummer criterion is pointwise.  For the later weight calculation we will need
rational branches as actual algebraic charts carrying regular maps to
$G$.

Let $X$ be an $F$-variety and let $a:X\to A$ be a morphism.  The ordinary
square-root cover is
\begin{equation}\label{eq:square-root-torsor}
 \cS(a)=X\times_{A,[2]}A
 =\{(x,t):t^2=a(x)\}
\end{equation}
a finite \'etale $A[2]$-torsor over $X$.  Put
\begin{equation}\label{eq:cartan-branch-kernel}
 \mathfrak B_A=
 \ker\left(H^1(F,A[2])\longrightarrow H^1(F,H)\right)
\end{equation}
For $\xi\in\mathfrak B_A$, choose a cocycle $c=(c_\sigma)$ representing
$\xi$.  Since its image in $H^1(F,H)$ is trivial, choose $h\in H(\Fb)$ with
\[
 h^{-1}\sigma(h)=c_\sigma
\]
Over $\Fb$ keep the same square-root cover but change its Galois action to
\begin{equation}\label{eq:twisted-square-root-action}
 \sigma *_c(x,t)=(\sigma x,c_\sigma\sigma(t))
\end{equation}
The resulting $F$-form is denoted
\[
 \pi_\xi:\cS_\xi(a)\longrightarrow X
\]
and is called the square-root branch attached to $\xi$.  Changing the cocycle
within its cohomology class gives an isomorphic $F$-form.

\begin{proposition}[Twisted Cartan-lift charts]\label{prop:twisted-cartan-charts}
Assume that $A$ is $F$-split.  For every $\xi\in\mathfrak B_A$ the formula
\begin{equation}\label{eq:twisted-u-formula}
 u_\xi(x,t)=ht
\end{equation}
descends to a regular $F$-morphism $u_\xi:\cS_\xi(a)\to G$ satisfying
\begin{equation}\label{eq:twisted-cartan-lift}
 \tau(u_\xi)=a\circ\pi_\xi
\end{equation}
Moreover:
\begin{enumerate}[label=\textup{(\alph*)},leftmargin=2.3em]
\item For $x\in X(F)$ one has
      \begin{equation}\label{eq:twisted-fiber-kummer}
       \cS_{\xi,x}(a)(F)\ne\varnothing
       \quad\Longleftrightarrow\quad
       \delta(a(x))=\xi
\end{equation}
      and therefore
      \begin{equation}\label{eq:liftable-locus-decomposition}
       \{x\in X(F):a(x)\in\tau(G(F))\}
       =\coprod_{\xi\in\mathfrak B_A}
          \pi_\xi\bigl(\cS_\xi(a)(F)\bigr)
\end{equation}
\item Changing the choice of $h$ changes $u_\xi$ by left multiplication by an
      element of $H(F)$
\item If $Q:X\to\g$ is an $F$-morphism with
      \[
       \theta(Q)=\Ad(a)Q
\]
      then
      \begin{equation}\label{eq:twisted-Y}
       Y_\xi=\Ad(u_\xi)Q:\cS_\xi(a)\longrightarrow\h
\end{equation}
      is an $F$-morphism and
      \begin{equation}\label{eq:twisted-inverse-conjugate}
       \Ad(u_\xi^{-1})Y_\xi=Q\circ\pi_\xi
\end{equation}
\end{enumerate}
\end{proposition}

\begin{proof}
The formula \eqref{eq:twisted-u-formula} is compatible with the twisted
Galois action because
\[
 u_\xi(\sigma *_c(x,t))
 =h c_\sigma\sigma(t)
 =\sigma(h)\sigma(t)
 =\sigma(ht)
\]
so it descends to $F$.  Since $h\in H(\Fb)$ and $t\in A(\Fb)$ one has
\[
 \tau(ht)=t^2=a(x)
\]
where $t^2=a(x)$ by the definition of $\cS(a)$.  This gives
\eqref{eq:twisted-cartan-lift}.

For $x\in X(F)$, an $F$-point of the twisted fiber is a square root
$t^2=a(x)$ satisfying
\[
 t=c_\sigma\sigma(t)
\]
Equivalently $t^{-1}\sigma(t)=c_\sigma$, so such a point exists exactly when
$\delta(a(x))=\xi$.  This proves \textup{(a)} by
\cref{thm:kummer-criterion}.  If $h'$ is another trivialization of the same
cocycle then $h'h^{-1}\in H(F)$, proving \textup{(b)}.  Finally, \textup{(c)} follows from the elementary conjugation identity
recorded later as \cref{lem:formal-descent}; it is stated separately there only
for convenient reuse.  The inverse-conjugate identity is immediate from the
definition.
\end{proof}

\begin{remark}
For the non-Archimedean local fields considered in the trace formula,
$\mathfrak B_A$ is finite.  Thus the rational Cartan-lift locus is covered by
a finite family of algebraic twists, not by a pointwise choice of square
roots.  The twists are geometrically isomorphic, but their sets of $F$-points
can of course be different.
\end{remark}

\begin{remark}
The rank-one norm conditions of \cref{lem:sl2-section} describe the image of
the corresponding root-$\SL_2$ section.  The parameters $\gamma_\beta$ are
therefore not harmless normalizing constants: their square classes determine
which coroot values occur over $F$ through this explicit rank-one
construction.  They need not by themselves determine the image of the full
Cartan map, which is governed by \cref{thm:kummer-criterion}.  When
$H=\GL_n$, one has $H^1(F,H)=1$, so the global Kummer
obstruction vanishes and the square classes still label the rational
Cartan-lift branches.  In the explicit symplectic family of Section~10, a
square-class parameter also controls the limiting ambient $G(F)$-nilpotent
orbit; no canonical identification of these two kinds of branch data is
intended.
\end{remark}

\subsection{The coroot isogeny and its descent obstruction}
Let $\cB=\{\beta_1,\ldots,\beta_r\}$ be a full-rank strongly orthogonal
system.  It is useful to separate the two maps already implicit in
\cref{thm:orthogonal-section}.  Put
\[
 A_\cB=\prod_{\beta\in\cB}\Gm
\]
and define
\[
 N_\cB:\mathcal R_\cB\longrightarrow A_\cB,
 \qquad
 N_\cB((w_\beta)_\beta)
 =\bigl(N_{E_\beta/F}(w_\beta)\bigr)_\beta
\]
and
\[
 \overline\nu_\cB:A_\cB\longrightarrow A,
 \qquad
 \overline\nu_\cB((z_\beta)_\beta)
 =\prod_{\beta\in\cB}\beta^\vee(z_\beta)
\]
Then $\nu_\cB=\overline\nu_\cB\circ N_\cB$.  Since $\cB$ has full rank,
$\overline\nu_\cB$ is an isogeny.  Write
\[
 K_\cB=\ker\overline\nu_\cB
\]
The exact sequence
\[
 1\longrightarrow K_\cB\longrightarrow A_\cB
 \xrightarrow{\overline\nu_\cB}A\longrightarrow1
\]
and Hilbert 90 give
\[
 A(F)/\overline\nu_\cB(A_\cB(F))
 \simeq H^1(F,K_\cB)
\]
Thus the failure of the strongly orthogonal coroots to generate $A(F)$ is
the connecting obstruction associated with the finite kernel
$K_\cB$.

The preceding norm maps determine which points of $A_\cB(F)$ are reached by
the explicit rank-one sections.  After passing through the coroot isogeny,
\cref{thm:kummer-criterion} determines which remaining points of $A(F)$ admit
a lift through the Cartan map.  The two obstructions are therefore
separate: the first comes from the norm maps and the coroot isogeny, while
the second is the intrinsic symmetric-space obstruction measured in
$H^1(F,H)$.  The twisted square-root charts of
\cref{prop:twisted-cartan-charts} resolve this second obstruction uniformly;
the strongly orthogonal construction gives explicit representatives when the
additional norm and isogeny obstructions vanish.

\section{\texorpdfstring{The Jacobi family and $\theta$-descent}{The Jacobi family and theta-descent}}

\subsection{Reduction to the adjoint derived group}
The root construction belongs to the derived Lie algebra.

\begin{proposition}\label{prop:adjoint-reduction}
Let $G$ be connected reductive.  Write
$\g=\mathfrak z(\g)\oplus\g_{\mathrm{der}}$.  Then
\[
 \g/\!/G\simeq
 \mathfrak z(\g)\times
 (\g_{\mathrm{der}}/\!/G_{\mathrm{ad}})
\]
The Jacobi family, regularity, and the principal nilpotent degeneration are
contained in $\g_{\mathrm{der}}$ and are geometrically unchanged, after
identifying the Lie algebras, under a central isogeny of semisimple groups.
Rational lifts and rational orbit decompositions can nevertheless change.
Moreover, for a $\theta$-equivariant central isogeny
$\pi:G_1\to G_2$,
\[
 \pi\circ\tau_{G_1}=\tau_{G_2}\circ\pi
\]
Thus the geometric construction may be carried out for the adjoint derived
group.  Pulling an $F$-rational Cartan lift back through $\pi$ introduces
only the usual obstruction in $H^1(F,\ker\pi)$.
\end{proposition}

\begin{proof}
The adjoint action is trivial on $\mathfrak z(\g)$ and factors through the
adjoint group on $\g_{\mathrm{der}}$, giving the product decomposition of
the invariant quotient.  A central isogeny of semisimple groups has finite
kernel, so its differential is an isomorphism in characteristic zero and
identifies the root spaces, adjoint quotients, regular loci, and geometric
nilpotent orbits.  The identity for $\tau$ follows from $\theta$-equivariance of
$\pi$.  Finally, the obstruction to lifting an $F$-point through a central
isogeny is the connecting class in $H^1(F,\ker\pi)$.
\end{proof}

The center decomposes into its $+1$- and $-1$-eigenspaces under $\theta$.  These
coordinates are independent of the root construction and may be restored
after treating the derived algebra.  We therefore work with the adjoint
semisimple group in the rest of this section; the central lifting obstruction
is included with the other rational branch data.

\subsection{Definition of the family}
Assume henceforth in this section that $G$ is adjoint.  Fix the $F$-split $\theta$-split maximal torus $A$ used in Section~5, a
Borel subgroup containing $A$, and a set of simple roots $\Delta$.  Choose root vectors
\[
 n_\alpha\in\g_\alpha,
 \qquad n_{-\alpha}\in\g_{-\alpha}
\]
from a compatible Chevalley basis, so that
$[n_\alpha,n_{-\alpha}]=\alpha^\vee$, and normalize them further so that
\begin{equation}\label{eq:theta-root-vectors}
 \theta(n_\alpha)=\gamma_\alpha n_{-\alpha}
 \qquad
 \theta(n_{-\alpha})=\gamma_\alpha^{-1}n_\alpha
\end{equation}
for constants $\gamma_\alpha\in F^\times$.  If the compatible root
vectors are replaced by
$n_\alpha' =c_\alpha n_\alpha$ and
$n_{-\alpha}'=c_\alpha^{-1}n_{-\alpha}$, then
$\gamma_\alpha$ is replaced by $c_\alpha^2\gamma_\alpha$.
Consequently its square class is unchanged by such compatible rescaling.

\begin{definition}\label{def:jacobi-family}
For $q=(q_\alpha)_{\alpha\in\Delta}$, define
\begin{equation}\label{eq:jacobi-family}
 Q(q)=\sum_{\alpha\in\Delta}
 \left(\gamma_\alpha n_{-\alpha}+q_\alpha n_\alpha\right)
\end{equation}
We call $\cQ=\{Q(q)\}$ the $\theta$-adapted Jacobi family.  The point
\[
 Q(0)=\sum_{\alpha\in\Delta}\gamma_\alpha n_{-\alpha}
\]
is principal nilpotent.
\end{definition}

This is a fixed-negative-coefficient slice in the generalized Toda lattice of
Kostant \cite{KostantToda}.  Two independent properties make it useful here:
its quotient map has the expected rank on the fixed part of the adjoint
quotient, and it carries a torus-valued $\theta$-twist.

\subsection{The twisting element}
Because $G$ is adjoint, the simple roots form a basis of $X^*(A)$.  For
$q\in\Gm^\Delta$, let $a(q)\in A$ be the unique element satisfying
\begin{equation}\label{eq:a-q}
 \alpha(a(q))=q_\alpha^{-1}
 \qquad(\alpha\in\Delta)
\end{equation}

\begin{lemma}\label{lem:theta-Q}
For every $q\in\Gm^\Delta$,
\begin{equation}\label{eq:theta-Q}
 \theta(Q(q))=\Ad(a(q))Q(q)
\end{equation}
\end{lemma}

\begin{proof}
By \eqref{eq:theta-root-vectors},
\[
 \theta(Q(q))
 =\sum_{\alpha\in\Delta}
 \left(n_\alpha+q_\alpha\gamma_\alpha n_{-\alpha}\right)
\]
On the other hand,
\begin{align*}
 \Ad(a(q))Q(q)
 &=\sum_{\alpha\in\Delta}
 \left(
 \gamma_\alpha\alpha(a(q))^{-1}n_{-\alpha}
 +q_\alpha\alpha(a(q))n_\alpha
 \right) \\
 &=\sum_{\alpha\in\Delta}
 \left(q_\alpha\gamma_\alpha n_{-\alpha}+n_\alpha\right)
\end{align*}
\end{proof}

\begin{lemma}[$\theta$-descent]\label{lem:formal-descent}
Let $X\in\g(\Fb)$ and $a,u\in G(\Fb)$ satisfy
\[
 \theta(X)=\Ad(a)X
 \qquad
 \tau(u)=a
\]
Then $\Ad(u)X\in\h(\Fb)$.  If $X$, $a$, and $u$ are defined over $F$, then
$\Ad(u)X\in\h(F)$.
\end{lemma}

\begin{proof}
We compute
\[
 \theta(\Ad(u)X)
 =\Ad(\theta(u))\theta(X)
 =\Ad(\theta(u)a)X
 =\Ad(u)X
\]
\end{proof}

\begin{proposition}\label{prop:jacobi-descent}
Let $q\in(F^\times)^\Delta$.  If there exists $u(q)\in G(F)$ with
$\tau(u(q))=a(q)$, then
\[
 \Ad(u(q))Q(q)\in\h(F)
\]
Such a rational lift exists if and only if
\[
 \delta(a(q))\in
 \ker\left(H^1(F,A[2])\longrightarrow H^1(F,H)\right)
\]
\end{proposition}

\begin{proof}
The first assertion is \cref{lem:theta-Q,lem:formal-descent}; the second is
\cref{thm:kummer-criterion}.
\end{proof}

\subsection{Generic transversality to the adjoint quotient}
Let $\chi:\g\to\cC_G=\g/\!/G$ be the adjoint quotient and let
$m_1,\ldots,m_r$ be the exponents of $\g$.  We first record two elementary
lemmas used in the differential calculation.

\begin{lemma}[The differential of the adjoint quotient]
\label{lem:quotient-differential}
Let $x\in\g(\Fb)$ be regular semisimple, and let $B$ be a nondegenerate
invariant bilinear form on $\g$.  Then
\[
 \g=[\g,x]\oplus\g_x
 \qquad
 \ker(d\chi_x)=[\g,x]
\]
If $V\subset\g$ is a linear subspace and
$\operatorname{pr}_x:\g\to\g_x$ is projection along $[\g,x]$, then
\[
 \rank\bigl(d\chi_x|_V\bigr)=\dim\operatorname{pr}_x(V)
\]
\end{lemma}

\begin{proof}
Invariance of $B$ gives
\[
 [\g,x]^\perp=\g_x
\]
Since $x$ is semisimple, $\ad(x)$ is semisimple, so its image and kernel are
complementary; this proves the first decomposition.  Invariant functions are
constant on the orbit of $x$, hence
$[\g,x]\subseteq\ker(d\chi_x)$.  The adjoint quotient is smooth at every
regular element and has relative dimension $\dim\g-r$.  Both spaces therefore
have dimension $\dim\g-r$, proving equality.  The final assertion follows by
passing to the quotient $\g/[\g,x]\simeq\g_x$.
\end{proof}

\begin{lemma}[The principal rotation coefficient]
\label{lem:principal-rotation}
Let $m\geq1$ and let $V_{2m}=\operatorname{Sym}^{2m}(\Fb^2)$ be the
irreducible $\slie_2$-module of highest weight $2m$.  Write
\[
 v_2=X^{m+1}Y^{m-1}
 \qquad
 v_0=X^mY^m
\]
for nonzero vectors of weights $2$ and $0$.  Let $s$ be a projective
$\SL_2$ element which conjugates $e+f$ to $h$ and whose action on binary forms,
up to a nonzero scalar, is
\[
 p(X,Y)\longmapsto p(X+Y,Y-X)
\]
The coefficient of $v_0$ in $s v_2$ is zero when $m$ is even.  If
$m=2k+1$, it is a nonzero scalar multiple of
\[
 2(-1)^k\binom{2k}{k}
\]
In particular, the projection of $sV_{2m}(2)$ to $V_{2m}(0)$ is nonzero if
and only if $m$ is odd.
\end{lemma}

\begin{proof}
The required coefficient is, up to the scalar used to lift $s$ from
$\operatorname{PGL}_2$,
\[
 C_m=[X^mY^m](X+Y)^{m+1}(Y-X)^{m-1}
\]
Setting $Y=1$ gives
\[
 C_m=[z^m](1+z)^{m+1}(1-z)^{m-1}
     =[z^m](1+z)^2(1-z^2)^{m-1}
\]
If $m=2k$, only the terms $1$ and $z^2$ in $(1+z)^2$ contribute, and
\[
 C_{2k}=(-1)^k\binom{2k-1}{k}
       +(-1)^{k-1}\binom{2k-1}{k-1}=0
\]
If $m=2k+1$, only the term $2z$ contributes, so
\[
 C_{2k+1}=2(-1)^k\binom{2k}{k}
\]
This is nonzero in characteristic zero.
\end{proof}

\begin{theorem}\label{thm:jacobi-rank}
There is a point $q^0\in(\Fb^\times)^\Delta$ at which
\begin{equation}\label{eq:jacobi-differential-rank}
 \rank d(\chi|_{\cQ})_{q^0}
 =\#\{i:m_i\text{ is odd}\}
\end{equation}
\end{theorem}

\begin{proof}
Put $f=Q(0)$.  Take the usual principal element $h=2\rho^\vee$ and write
\[
 2\rho^\vee=\sum_{\alpha\in\Delta}c_\alpha\alpha^\vee
 \qquad c_\alpha>0
\]
With the Chevalley normalization above,
\[
 e=\sum_{\alpha\in\Delta}
      c_\alpha\gamma_\alpha^{-1}n_\alpha
\]
satisfies $[e,f]=h$, $[h,e]=2e$, and $[h,f]=-2f$.  Thus $(e,h,f)$ is the
principal triple adapted to the chosen Borel.  Put
$q^0_\alpha=c_\alpha\gamma_\alpha^{-1}$; every $q^0_\alpha$ is nonzero and
$Q(q^0)=e+f$.  Inside the principal $\SL_2$, the element $e+f$ is conjugate
to $h$; choose $s$ as in \cref{lem:principal-rotation}.

The principal grading decomposes $\mathfrak{g}$ by $h$'s eigenvalues:
\[
 T_{Q(q^0)}\cQ=\bigoplus_{\alpha\in\Delta}\g_\alpha=\g(2)
\]
One can also decompose $\mathfrak{g}$ into the irreducible representations of this principal $SL_2$: 
\[
 \g=\bigoplus_{i=1}^r V_{2m_i}
\]
Each summand has a one-dimensional weight-two space and a one-dimensional
weight-zero space, and
\[
 \g_h=\g(0)=\bigoplus_{i=1}^r V_{2m_i}(0)
\]
By \cref{lem:quotient-differential}, the rank of the restricted differential
is the dimension of the projection of $s\g(2)$ to $\g(0)$.  The principal
$\SL_2$ preserves every summand $V_{2m_i}$, so this projection is diagonal
with respect to the preceding direct sums.  By
\cref{lem:principal-rotation}, the contribution of $V_{2m_i}$ has rank one
exactly when $m_i$ is odd.  This proves
\eqref{eq:jacobi-differential-rank}.
\end{proof}

Let $\vartheta$ be the involution induced by $\theta$ on
$\cC_G=\g/\!/G$.  Choose homogeneous generators of $F[\g]^G$ of degrees
$m_i+1$.  Since $\theta$ acts as $-1$ on the Cartan $\Lie(A)$, the Chevalley
restriction theorem gives
\[
 \vartheta^*P=(-1)^{\deg P}P
\]
for every homogeneous invariant polynomial $P$.  Denote the fixed-point
locus of $\vartheta$ by $\cC_G^\vartheta$.

\begin{lemma}\label{lem:odd-exponents-rank-h}
For a split symmetric pair,
\[
 \#\{i:m_i\text{ is odd}\}=\rank H
\]
Consequently,
\[
 \dim\cC_G^\vartheta=\rank H
\]
\end{lemma}

\begin{proof}
The assertion is additive under products.  For a simple factor it follows
from the classification in \cref{tab:classification} and the standard lists
of exponents.  In type $A_{n-1}$ the odd exponents among
$1,\ldots,n-1$ number $\lfloor n/2\rfloor$; in types $B_n$ and $C_n$ all
$n$ exponents are odd.  In type $D_n$ the number is
$2\lfloor n/2\rfloor$.  For $E_6,E_7,E_8,F_4,G_2$ the numbers are,
respectively, $4,7,8,4,2$.  These are the ranks of the fixed
algebras in \cref{tab:classification}.  The second assertion follows because
$\cC_G^\vartheta$ is obtained by setting the odd-degree basic invariants equal
to zero.
\end{proof}

\begin{corollary}[The $\theta$-fixed part of the quotient]
\label{cor:jacobi-fixed-quotient}
The restriction $\chi|_{\cQ}$ factors through $\cC_G^\vartheta$, and
\[
 \cQ\longrightarrow\cC_G^\vartheta
\]
is dominant and generically smooth of relative dimension
$\rank G-\rank H$.  It is generically \'etale exactly in the equal-rank cases.
\end{corollary}

\begin{proof}
By \cref{lem:theta-Q}, $Q(q)$ and $\theta(Q(q))$ are $G$-conjugate.  Their
images in the adjoint quotient are therefore equal, so $\chi|_{\cQ}$ lands in $\cC_G^\vartheta$.  The target has dimension $\rank H$ by
\cref{lem:odd-exponents-rank-h}, and \cref{thm:jacobi-rank} shows that the
differential has this rank at $q^0$.  The source has dimension $\rank G$.
The remaining assertions follow.
\end{proof}

Equivalently, every homogeneous invariant of odd degree vanishes on $\cQ$.
In the non-equal-rank types this is not an obstruction to the fixed-side
construction: it identifies the correct quotient.  The family has dense image in the
fixed part $\cC_G^\vartheta$, rather than in all of $\cC_G$.

\subsection{\texorpdfstring{From the $G$-quotient to the $H$-quotient}{From the G-quotient to the H-quotient}}

The inclusion $\h\hookrightarrow\g$ induces a morphism
\[
 \pi_H:\h/\!/H\longrightarrow \cC_G^\vartheta
\]
We first study this map geometrically.  After base change to $\Fb$, the
squaring map
\[
 [2]:A_{\Fb}\longrightarrow A_{\Fb}
\]
is surjective.  Hence, for every $q\in(\Fb^\times)^\Delta$, we may choose
$t\in A(\Fb)$ with
\[
 t^2=a(q)
\]
Since $A$ is $\theta$-split, $\theta(t)=t^{-1}$, and therefore
\[
 \tau(t)=\theta(t)^{-1}t=t^2=a(q)
\]
By \cref{lem:theta-Q,lem:formal-descent},
\[
 Y(q):=\Ad(t)Q(q)\in\h(\Fb)
\]
and $Y(q)$ is $G(\Fb)$-conjugate to $Q(q)$.  Thus
\[
 \chi_G(Y(q))=\chi_G(Q(q))
\]
The points $\chi_G(Q(q))$ are Zariski dense in $\cC_G^\vartheta$ by
\cref{cor:jacobi-fixed-quotient}.  It follows that
\[
 (\h/\!/H)_{\Fb}\longrightarrow(\cC_G^\vartheta)_{\Fb}
\]
is dominant, and therefore so is
\[
 \pi_H:\h/\!/H\longrightarrow\cC_G^\vartheta
\]
over $F$.

Both sides have dimension $\rank H$, so $\pi_H$ is generically finite.  In
the equal-rank case $\cC_G^\vartheta=\cC_G$, and the Jacobi family gives a
generically \'etale multisection of the full adjoint quotient.

This dominance argument is geometric: over $\Fb$, a square root of $a(q)$
always exists and produces an element of $\h$ with the same $G$-invariants as
$Q(q)$.  Rational descent is a separate question.  Over $F$, the Kummer class
of $a(q)$ determines whether such a Cartan lift exists and, when it does, the
twisted square-root branch on which it is realized; see
\cref{thm:kummer-criterion,prop:twisted-cartan-charts}.

The two rank defects appearing above agree:
\begin{equation}\label{eq:two-rank-defects}
\dim\cQ-\dim\cC_G^\vartheta
=\rank G-\rank H
=\dim(T_H\backslash T_G)
\end{equation}
where $T_H\subset T_G$ are the regular centralizer tori of
\cref{eq:residual-torus-dimension}. Thus, in the non-equal-rank case, the
relative dimension of the Jacobi map agrees with the geometric dimension of
the residual centralizer torus occurring in the fixed-side orbital integral.
When this torus is $F$-split, the same number is the rank of its valuation
lattice.  Equation \eqref{eq:two-rank-defects} is only a numerical comparison;
the detailed interaction of the two integrations belongs to the full
fixed-side expansion.

\begin{example}[The $A_2$ descendant]\label{ex:A2-descendant}
For $(\slie_3,\solie_3)$ take the standard simple-root vectors and
$\gamma_{\alpha_1}=\gamma_{\alpha_2}=1$.  Then
\[
 Q(q_1,q_2)=
 \begin{pmatrix}
 0&q_1&0\\
 1&0&q_2\\
 0&1&0
 \end{pmatrix}
 \qquad
 \det(\lambda-Q)=\lambda^3-(q_1+q_2)\lambda
\]
The adjoint quotient of $\slie_3$ has basic invariants of degrees $2$ and
$3$, while the cubic invariant vanishes on $\solie_3$.  Hence
$\cC_G^\vartheta$ is the degree-two line, and $\chi|_{\cQ}$ has one-dimensional
fibers over it.  For a common regular element,
$\dim(T_H\backslash T_G)=2-1=1$.  Thus the simplest non-equal-rank
descendant realizes the numerical equality in
\eqref{eq:two-rank-defects}: the one-dimensional fiber of the Jacobi map
matches the geometric dimension of the residual centralizer torus.
\end{example}

\section{Jacobi descent charts and logarithmic quotient coordinates}
\label{sec:jacobi-logarithmic}

We now assemble the preceding results in the equal-rank case and prove the
measure identity that makes the Jacobi parameters particularly well adapted
to the fixed-side calculation.  Throughout this section $F$ is
non-Archimedean of characteristic zero and residue characteristic different
from $2$, and $(G,H,\theta)$ is an equal-rank split symmetric pair.  We work
first with the adjoint semisimple derived group and the $F$-split
$\theta$-split maximal torus $A$ fixed in Section~5.  The harmless central
modifications are as in \cref{prop:adjoint-reduction}.

The principal observation is that the square-root cover introduced by the
Cartan lift is not merely a device for rational descent.  It linearizes the
Jacobi family itself.  Recall from \eqref{eq:a-q} that $a(q)\in A$ is
characterized by $\alpha(a(q))=q_\alpha^{-1}$.  If $t^2=a(q)$, write
$z=(z_\alpha)_{\alpha\in\Delta}$ with
\begin{equation}\label{eq:z-square-root-coordinates}
 z_\alpha=\alpha(t)^{-1}
 \qquad q_\alpha=z_\alpha^2
\end{equation}
Then with
\[
 x_\alpha=n_\alpha+\gamma_\alpha n_{-\alpha}\in\h
\]
one has
\begin{equation}\label{eq:linear-fixed-jacobi}
 Y(z):=\Ad(t)Q(z^2)
 =\sum_{\alpha\in\Delta}z_\alpha x_\alpha
\end{equation}
Thus the $\theta$-fixed family is linear in the square-root Jacobi coordinates.
The determinant calculation below shows that these same coordinates make the
$H^\circ$-relative quotient density exactly logarithmic; the finite passage
to the full $H$-quotient is recorded separately below.

\subsection{Height parity and the relative determinant}

Work temporarily over $\Fb$.  Choose root vectors $e_\beta\in\g_\beta$ for
$\beta\in\Phi^+$ and put
\[
 f_\beta=\theta(e_\beta)\in\g_{-\beta},\qquad
 x_\beta=e_\beta+f_\beta\in\h,\qquad
 y_\beta=e_\beta-f_\beta\in\p
\]
For the simple roots we take $e_\alpha=n_\alpha$, so that
$f_\alpha=\gamma_\alpha n_{-\alpha}$ and \eqref{eq:linear-fixed-jacobi}
has the displayed form.

Let
\[
 \Phi_{\mathrm{ev}}^+
 =\{\beta\in\Phi^+:\operatorname{ht}(\beta)\text{ is even}\}
\]
and
\[
 \Phi_{\mathrm{odd}}^{+,*}
 =\{\gamma\in\Phi^+:\operatorname{ht}(\gamma)>1
       \text{ is odd}\}
\]
Put
\[
 E_\h=\bigoplus_{\beta\in\Phi_{\mathrm{ev}}^+}\Fb x_\beta,
 \qquad
 O_\h=\bigoplus_{\gamma\in\Phi_{\mathrm{odd}}^{+,*}}\Fb x_\gamma
\]
and define $E_\p,O_\p$ in the same way using $y_\beta,y_\gamma$.  Finally put
\[
 V_\p=\bigoplus_{\alpha\in\Delta}\Fb y_\alpha
\]

\begin{lemma}[Parity count]\label{lem:height-parity-count}
In the equal-rank case all exponents $m_i$ of $\g$ are odd.  If
\[
 N=\abs{\Phi_{\mathrm{ev}}^+}
\]
then
\[
 \abs{\Phi_{\mathrm{odd}}^{+,*}}=N
\]
and
\begin{equation}\label{eq:height-parity-sum}
 \sum_{\gamma\in\Phi_{\mathrm{odd}}^{+,*}}
       \operatorname{ht}(\gamma)
 -\sum_{\beta\in\Phi_{\mathrm{ev}}^+}
       \operatorname{ht}(\beta)
 =N
\end{equation}
\end{lemma}

\begin{proof}
By \cref{lem:odd-exponents-rank-h}, equal rank means that every exponent is
odd.  The number of positive roots of height $j$ is
$\#\{i:m_i\geq j\}$.  Indeed, under the principal element
$h=2\rho^\vee$, a root space of height $j$ has weight $2j$, while the
summand $V_{2m_i}$ contains weight $2j$ exactly when $m_i\geq j$.  An exponent $m_i=2k+1$ contributes $k$ even heights and
$k+1$ odd heights, the extra odd height being $1$.  Removing the $r$ simple
roots therefore leaves equally many odd and even roots.  The contribution of
this exponent to the left side of \eqref{eq:height-parity-sum} is
\[
 (3+5+\cdots+(2k+1))-(2+4+\cdots+2k)=k
\]
which is also its contribution to $N$.
\end{proof}

Fix bases in the preceding root spaces and set
\begin{equation}\label{eq:parity-minor}
 R(z)=\det\left(
 \ad Y(z):O_\h\longrightarrow E_\h
 \right)
\end{equation}
The determinant depends on the chosen root-vector normalizations only by a
nonzero constant.

\begin{lemma}[Comparison of the parity blocks]\label{lem:parity-block-comparison}
Under the identifications $x_\beta\leftrightarrow y_\beta$ and
$x_\gamma\leftrightarrow y_\gamma$, the matrices of
\[
 \ad Y:O_\h\longrightarrow E_\h
 \qquad\text{and}\qquad
 \ad Y:O_\p\longrightarrow E_\p
\]
are identical, up to the fixed choices of bases in the root lines.  In
particular their determinants differ by a nonzero constant independent of
$z$.
\end{lemma}

\begin{proof}
It is enough to compare one simple-root summand $x_\alpha$ of $Y$ on one
positive odd root $\gamma$ of height greater than one.  If
$\beta=\gamma+\alpha$ is a root and
$[e_\alpha,e_\gamma]=N_{\alpha,\gamma}e_\beta$, then applying $\theta$ gives
$[f_\alpha,f_\gamma]=N_{\alpha,\gamma}f_\beta$.  Hence the $\beta$-components
of the two brackets are
\[
 [x_\alpha,x_\gamma]_\beta
   =N_{\alpha,\gamma}x_\beta
 \qquad
 [x_\alpha,y_\gamma]_\beta
   =N_{\alpha,\gamma}y_\beta
\]
If instead $\beta=\gamma-\alpha$ is a positive root and
$[f_\alpha,e_\gamma]=M_{\alpha,\gamma}e_\beta$, applying $\theta$ gives
$[e_\alpha,f_\gamma]=M_{\alpha,\gamma}f_\beta$, and again
\[
 [x_\alpha,x_\gamma]_\beta
   =M_{\alpha,\gamma}x_\beta
 \qquad
 [x_\alpha,y_\gamma]_\beta
   =M_{\alpha,\gamma}y_\beta
\]
These are the only possible even-height root components.  Summing with
coefficients $z_\alpha$ proves the assertion.
\end{proof}

\begin{lemma}[The relative parity determinant]\label{lem:relative-parity-determinant}
With the preceding notation,
\begin{equation}\label{eq:p-parity-determinant}
 \det\left(
 \ad Y(z):V_\p\oplus O_\p
 \longrightarrow \Lie(A)\oplus E_\p
 \right)
 =c\left(\prod_{\alpha\in\Delta}z_\alpha\right)R(z)
\end{equation}
for a nonzero constant $c$.  Consequently, on the common regular-semisimple locus,
\begin{equation}\label{eq:relative-discriminant-parity}
 D_H^G(Y(z))
 =c'\left[
       \left(\prod_{\alpha\in\Delta}z_\alpha\right)R(z)
      \right]^2
\end{equation}
\end{lemma}

\begin{proof}
Since $Y$ is a sum of simple-root fixed vectors, $\ad Y$ changes root-height
parity.  Order the source in \eqref{eq:p-parity-determinant} as
$V_\p\oplus O_\p$ and the target as $\Lie(A)\oplus E_\p$. Only the simple vector $y_\alpha$ can contribute to the Cartan component through its bracket with $x_\alpha$, and
\[
[x_\alpha,y_\alpha]
=-2[e_\alpha,f_\alpha]
=-2\gamma_\alpha\alpha^\vee
\]
Hence the $V_\p\to\Lie(A)$ block is, up to the fixed diagonal factors
$-2\gamma_\alpha$, the differential of the simple-coroot isogeny
\[
\prod_{\alpha\in\Delta}\Gm\longrightarrow A
\qquad
(t_\alpha)_{\alpha\in\Delta}\longmapsto
\prod_{\alpha\in\Delta}\alpha^\vee(t_\alpha)
\]
followed by $\diag(z_\alpha)$. In simple-root coordinates its fixed factor is
the Cartan matrix. In particular, its determinant is a nonzero constant times
\[
\prod_{\alpha\in\Delta} z_\alpha
\]
Thus the $V_\p\to\Lie(A)$ block is a fixed invertible Cartan matrix times
$\diag(z_\alpha)$.  By \cref{lem:parity-block-comparison}, the
$O_\p\to E_\p$ block has determinant equal to a fixed nonzero constant times
$R(z)$.  The full matrix is block triangular, proving
\eqref{eq:p-parity-determinant}.

For a common regular element $Y$ of $\h$ and $\g$, equal rank gives
$\g_Y=\h_Y$ and hence $\p_Y=0$.  Therefore
\[
 D_H^G(Y)=\det(\ad Y;\p)
\]
The decomposition
\[
 \p=(V_\p\oplus O_\p)\oplus(\Lie(A)\oplus E_\p)
\]
is exchanged by $\ad Y$.  Invariance of the fixed bilinear form identifies
the two off-diagonal blocks as transposes up to constant Gram matrices.
Their determinants therefore agree up to a nonzero constant, giving
\eqref{eq:relative-discriminant-parity}.
\end{proof}

\subsection{The ambient Jacobi Jacobian}

Choose homogeneous basic invariants $P_1,\ldots,P_r$ of $\g$.  Their degrees
$m_i+1$ are all even.  Put
\begin{equation}\label{eq:ambient-jacobi-jacobian}
 J_G(q)=\det\left(
 \frac{\partial P_i(Q(q))}{\partial q_\alpha}
 \right)_{i,\alpha}
\end{equation}

We use Kostant's regularity theorem for the principal Hessenberg space.  With
the Borel chosen in Section~7 and $f=Q(0)$, every element of
$f+\mathfrak b$ is regular; equivalently, the principal Hessenberg affine
space lies in $\g_{\mathrm{reg}}$ \cite{KostantToda}.  In particular,
\begin{equation}\label{eq:jacobi-everywhere-regular}
 Q(q)\in\g_{\mathrm{reg}}
 \qquad\text{for every }q\in\A^\Delta
\end{equation}

Let $\varepsilon_{\mathrm{ht}}$ be the height-parity involution acting by
$(-1)^{\operatorname{ht}(\beta)}$ on $\g_\beta$ and trivially on $\Lie(A)$.
It exists because $G$ is adjoint and the simple roots form a basis of
$X^*(A)$.  Write $\g_{\mathrm{odd}}$ and $\g_{\mathrm{even}}$ for its
$-1$ and $+1$ eigenspaces.  Since $Q(q)$ is odd,
\[
 \ad Q(q):\g_{\mathrm{odd}}\longrightarrow\g_{\mathrm{even}}
\]
The parity count gives
\begin{equation}\label{eq:g-parity-dimension}
 \dim\g_{\mathrm{odd}}=\dim\g_{\mathrm{even}}+r
\end{equation}

\begin{lemma}[A cofactor lemma]\label{lem:cofactor-kernel}
Let $R_0$ be a polynomial ring over a field and let
$B:R_0^{m+r}\to R_0^m$ be a polynomial matrix of full row rank at every
geometric point.  Suppose $v_1,\ldots,v_r$ are polynomial sections which form
a basis of $\ker B$ at every geometric point.  Fix the standard bases.  Under the corresponding coordinate Hodge-star
identification
$\bigwedge^m(R_0^{m+r})^*\simeq\bigwedge^rR_0^{m+r}$, the wedge of the
$m$ row covectors of $B$ maps to the Pl\"ucker cofactor vector formed from
its maximal minors, and
\[
 v_1\wedge\cdots\wedge v_r
\]
is a nonzero constant multiple of this vector.
\end{lemma}

\begin{proof}
Let $K=\ker B$.  Since the maximal minors of $B$ have no common geometric
zero, they generate the unit ideal.  Thus $B:R_0^{m+r}\to R_0^m$ is
surjective.  As the target is free, the sequence
\[
 0\longrightarrow K\longrightarrow R_0^{m+r}
 \xrightarrow{B}R_0^m\longrightarrow0
\]
splits, so $K$ is a projective module of rank $r$.  Taking determinant lines
gives the canonical identity
\[
 \det(R_0^{m+r})\simeq \det(K)\otimes\det(R_0^m)
\]
which, after the standard bases trivialize the two free determinant lines, is
the determinant-line form of the coordinate Hodge-star identification used
above.

The sections $v_1,\ldots,v_r$ define a map $R_0^r\to K$.  By hypothesis its
specialization at every geometric point is an isomorphism.  Its kernel and
cokernel therefore have empty support, hence both vanish; equivalently, the
$v_i$ form a global basis of $K$.  In particular
\[
 v_1\wedge\cdots\wedge v_r
\]
is a nowhere-vanishing generator of the determinant line $\bigwedge^rK$.

Now wedge the $m$ row covectors of $B$ and apply the coordinate Hodge star.
In the standard basis its coordinates are, up to the usual signs, the
maximal minors of $B$.  Because these minors generate the unit ideal, the
resulting Pl\"ucker cofactor vector is also nowhere vanishing.  At every
geometric point it spans the exterior power of the kernel of $B$, so it is a
second generator of $\bigwedge^rK$.  The ratio of these two generators is a
unit of $R_0$.  The only units in a polynomial ring over a field are the
nonzero constants, which proves the assertion.
\end{proof}

\begin{lemma}[Jacobi cofactor identity]\label{lem:jacobi-cofactor-identity}
With $R(z)$ as in \eqref{eq:parity-minor},
\begin{equation}\label{eq:ambient-jacobi-parity}
 J_G(z^2)=c\,R(z)^2
\end{equation}
for a nonzero constant $c$.
\end{lemma}

\begin{proof}
We first identify the kernel of the parity block uniformly in $q$.  Because the $P_i$ have even degree and $Q(q)$ is odd for the height-parity
involution, their gradients are odd.  Indeed, equivariance of the gradient
under the invariant bilinear form gives
\[
 \nabla P_i(\varepsilon_{\mathrm{ht}}Q)
 =\varepsilon_{\mathrm{ht}}\nabla P_i(Q)
\]
while homogeneity gives $\nabla P_i(-Q)=-\nabla P_i(Q)$.  Hence
\[
 \varepsilon_{\mathrm{ht}}\nabla P_i(Q(q))
   =-\nabla P_i(Q(q))
\]
The gradients are polynomial functions of $q$.  Since the $P_i$ are
invariant, each $\nabla P_i(Q(q))$ centralizes $Q(q)$.  By
\eqref{eq:jacobi-everywhere-regular}, $Q(q)$ is regular for every geometric
$q$, and the standard differential criterion for the adjoint quotient shows
that
\[
 \nabla P_1(Q(q)),\ldots,\nabla P_r(Q(q))
\]
form a basis of $\g_{Q(q)}$.  Thus the entire $r$-dimensional centralizer is
contained in $\g_{\mathrm{odd}}$.  For
\[
 B_Q:=\ad Q(q):\g_{\mathrm{odd}}\longrightarrow\g_{\mathrm{even}}
\]
the kernel is therefore exactly $\g_{Q(q)}$.  Together with
\eqref{eq:g-parity-dimension}, this gives
\[
 \dim\ker B_Q=r,
 \qquad
 \rank B_Q=\dim\g_{\mathrm{even}}
\]
so $B_Q$ is surjective for every geometric $q$, with the polynomial gradients
as a basis of its kernel.  Hence \cref{lem:cofactor-kernel} applies over the
polynomial ring $\Fb[q_\alpha:\alpha\in\Delta]$.

We next identify the relevant Pl\"ucker coordinate.  For this computation, write
\[
 b_\alpha=B(n_{-\alpha},n_\alpha)\ne0
\]
Since distinct root spaces are orthogonal unless their roots sum to zero,
the coefficient of $n_{-\alpha}$ in $\nabla P_i(Q(q))$ is
\[
 b_\alpha^{-1}\,
 dP_i|_{Q(q)}(n_\alpha)
 =b_\alpha^{-1}
   \frac{\partial P_i(Q(q))}{\partial q_\alpha}
\]
It follows that the coefficient of
\[
 \bigwedge_{\alpha\in\Delta}n_{-\alpha}
\]
in $\bigwedge_i\nabla P_i(Q(q))$ is a fixed nonzero constant times
$J_G(q)$.  Under the coordinate Hodge-star identification in
\cref{lem:cofactor-kernel}, this Pl\"ucker coordinate is the complementary
maximal minor of $B_Q$, namely the minor obtained by deleting the negative
simple-root columns.  Therefore
\begin{equation}\label{eq:ambient-cofactor-minor}
 J_G(q)=c_0 M_G(q),
 \qquad
 M_G(q)=\det\left(
 \ad Q(q):C\longrightarrow\g_{\mathrm{even}}
 \right)
\end{equation}
where
\[
 C=
 \left(\bigoplus_{\alpha\in\Delta}\g_\alpha\right)
 \oplus
 \left(\bigoplus_{\gamma\in\Phi_{\mathrm{odd}}^{+,*}}
       (\g_\gamma\oplus\g_{-\gamma})\right)
\]
This is an explicit identification of the ambient Jacobi Jacobian with a
single parity cofactor.

Now substitute $q=z^2$ and conjugate by $t$ from
\eqref{eq:z-square-root-coordinates}.  Since
$Y(z)=\Ad(t)Q(z^2)$ and both $C$ and $\g_{\mathrm{even}}$ are stable under
$\Ad(t)$, the corresponding matrices satisfy
\[
 B_Y
 =\Ad(t)|_{\g_{\mathrm{even}}}\circ
   B_Q\circ\Ad(t)^{-1}|_C,
 \qquad
 B_Y:=\ad Y(z)|_C
\]
The determinant of $\Ad(t)$ on $\g_{\mathrm{even}}$ is one: every nonzero
root space occurs together with its opposite, and $\Ad(t)$ is trivial on
$\Lie(A)$.  The same cancellation occurs on the non-simple summands of $C$.
On the positive simple-root line $\g_\alpha$, however,
$\Ad(t)$ acts by
\[
 \alpha(t)=z_\alpha^{-1}
\]
Consequently
\begin{equation}\label{eq:cofactor-conjugation}
 M_Y(z):=\det B_Y(z)
 =\left(\prod_\alpha z_\alpha\right)M_G(z^2)
\end{equation}

It remains to calculate $M_Y(z)$ from the $\theta$-eigenspace decomposition.
The non-simple odd root pairs in $C$ split into their $\h$- and $\p$-lines, so
\[
 C=O_\h\oplus C_2,
 \qquad
 C_2=
 \left(\bigoplus_{\alpha\in\Delta}\Fb e_\alpha\right)\oplus O_\p
\]
where
\[
 e_\alpha=(x_\alpha+y_\alpha)/2
\]
Likewise
\[
 \g_{\mathrm{even}}
 =E_\h\oplus\bigl(\Lie(A)\oplus E_\p\bigr)
\]
Because $Y\in\h$, the operator $\ad Y$ preserves the $\h$- and
$\p$-eigenspaces.  Moreover
\[
 \ad Y(O_\h)\subset E_\h
\]
a non-simple positive odd root cannot bracket with a simple root into the
Cartan, and the remaining possible root components have even height.  Thus,
with the displayed source and target orders, $B_Y(z)$ is block triangular.
Its first diagonal block is
\[
 \ad Y:O_\h\longrightarrow E_\h
\]
whose determinant is $R(z)$ by definition.

For the second diagonal block, projection to
$\Lie(A)\oplus E_\p$ kills the $x_\alpha/2$ component of each positive simple
vector $e_\alpha$ and retains its $y_\alpha/2$ component.  Thus, after the
fixed change of basis that rescales the $r$ simple columns by $2$, this block
is equivalent for the present determinant calculation to
\[
 \ad Y:V_\p\oplus O_\p
 \longrightarrow \Lie(A)\oplus E_\p
\]
By \cref{lem:relative-parity-determinant}, its determinant is
\[
 c_1\left(\prod_\alpha z_\alpha\right)R(z)
\]
Hence
\[
 M_Y(z)
 =c_2\left(\prod_\alpha z_\alpha\right)R(z)^2
\]
Combining this with \eqref{eq:cofactor-conjugation} gives
$M_G(z^2)=c_2R(z)^2$, and the identification
$J_G=c_0M_G$ in \eqref{eq:ambient-cofactor-minor} proves
\eqref{eq:ambient-jacobi-parity}.
\end{proof}

The preceding identity also shows that $R$ is nonzero, since
\cref{thm:jacobi-rank} makes $J_G$ nonzero in the equal-rank case.
Height combinatorics make $R$ explicit.

\begin{corollary}[Monomial Jacobi Jacobian]\label{cor:monomial-jacobi-jacobian}
Put
\begin{equation}\label{eq:kappa-root}
 \kappa=
 \sum_{\gamma\in\Phi_{\mathrm{odd}}^{+,*}}\gamma
 -\sum_{\beta\in\Phi_{\mathrm{ev}}^+}\beta
 =\sum_{\alpha\in\Delta}k_\alpha\alpha
\end{equation}
Then $k_\alpha\geq0$ and
\begin{equation}\label{eq:R-monomial}
 R(z)=c\prod_{\alpha\in\Delta}z_\alpha^{k_\alpha},
 \qquad
 J_G(q)=c'\prod_{\alpha\in\Delta}q_\alpha^{k_\alpha}
\end{equation}
\end{corollary}

\begin{proof}
Every nonzero matrix coefficient in the definition of $R(z)$ is a constant
times some $z_\alpha$ and can occur only between roots whose heights differ
by one.  A nonzero summand in the determinant expansion therefore chooses a
bijection from $\Phi_{\mathrm{odd}}^{+,*}$ to $\Phi_{\mathrm{ev}}^+$ such
that each chosen source-target pair $(\gamma,\beta)$ satisfies
$\abs{\operatorname{ht}(\gamma)-\operatorname{ht}(\beta)}=1$.  By
\eqref{eq:height-parity-sum}, the sum of
$\operatorname{ht}(\gamma)-\operatorname{ht}(\beta)$ over the $N$ chosen
pairs is $N$.  Each summand is at most one, so every one equals one.  Hence
for every chosen pair
\[
 \beta=\gamma-\alpha
\]
for a simple root $\alpha$.  The product of the corresponding matrix
coefficients therefore has exponent
\[
 \sum_{\gamma\in\Phi_{\mathrm{odd}}^{+,*}}\gamma
 -\sum_{\beta\in\Phi_{\mathrm{ev}}^+}\beta
 =\kappa
\]
and every nonzero determinant summand is a constant multiple of
$\prod_{\alpha\in\Delta}z_\alpha^{k_\alpha}$.
Since $R\neq0$, their total coefficient is nonzero and the exponents
$k_\alpha$ are nonnegative.  The second identity follows from
\eqref{eq:ambient-jacobi-parity} and $q_\alpha=z_\alpha^2$.
\end{proof}

\subsection{The logarithmic quotient density}

Let $H^\circ$ be the identity component of $H$.  The finite group
$H/H^\circ$ acts on $\h/\!/H^\circ$, and the natural map
$\h/\!/H^\circ\to\h/\!/H$ is finite.  In characteristic zero, after removing
the proper branch locus (and quotienting by any subgroup acting trivially),
this map is finite \'etale.  We shrink the common regular open set once and
for all to lie over this locus.  Passing between $H^\circ$ and $H$ then
changes the regular orbital calculation only by finite covering data.  On
sufficiently small $F$-analytic charts in this \'etale locus, the corresponding
Jacobian has locally constant nonzero absolute value.  We therefore compute
the differential density first for $H^\circ$ and record the finite-quotient
qualification in the descent-chart theorem below.

Choose homogeneous basic invariants
$s_1,\ldots,s_r$ of $\h$ for $H^\circ$ and write
\[
 \psi(z)=(s_1(Y(z)),\ldots,s_r(Y(z))),
 \qquad
 J_H(z)=\det\left(
 \frac{\partial s_i(Y(z))}{\partial z_\alpha}
 \right)
\]
The inclusion $\h\hookrightarrow\g$ induces a finite morphism
\[
 \pi:\h/\!/H^\circ\longrightarrow\g/\!/G
\]
Let $J_\pi$ denote its Jacobian in the chosen invariant coordinates.

\begin{lemma}[Relative quotient Jacobian]\label{lem:quotient-map-relative-jacobian}
On the common regular-semisimple locus,
\begin{equation}\label{eq:quotient-map-relative-jacobian}
 (J_\pi\circ\chi_H)^2=c\,D_H^G
\end{equation}
for a nonzero constant $c$ depending only on the choices of invariant
coordinates.
\end{lemma}

\begin{proof}
After base change to $\Fb$, choose a maximal torus $T$ of $H^\circ$.  Because
we are in the equal-rank case, it is also a maximal torus of $G$.  Put
$\tT=\Lie(T)$, and let $W_H$ and $W_G$ be the corresponding Weyl groups.
Chevalley restriction identifies the two geometric quotients with
\[
 \h/\!/H^\circ\simeq\tT/W_H
 \qquad\text{and}\qquad
 \g/\!/G\simeq\tT/W_G
\]
and the map $\pi$ is induced by the finite inclusion
$F[\tT]^{W_G}\subset F[\tT]^{W_H}$.

Choose linear coordinates $x_1,\ldots,x_r$ on $\tT$, and let
$s_1,\ldots,s_r$ and $P_1,\ldots,P_r$ be the chosen basic invariant
coordinates for $H^\circ$ and $G$.  The reflection-group Jacobian formula
gives
\[
 \det\left(\frac{\partial s_i}{\partial x_j}\right)
 =c_H\prod_{\alpha\in\Phi_H^+}\alpha
\]
and
\[
 \det\left(\frac{\partial P_i}{\partial x_j}\right)
 =c_G\prod_{\alpha\in\Phi_G^+}\alpha
\]
with $c_H,c_G\ne0$.  Since the $P_i$ are functions of the $s_j$, the chain
rule on $\tT$ yields
\[
 J_\pi(\chi_H(X))
 =c_0\,
 \frac{\prod_{\alpha\in\Phi_G^+}\alpha(X)}
      {\prod_{\alpha\in\Phi_H^+}\alpha(X)}
\]
on the common regular-semisimple locus.  Squaring and using the root-product formulas
for the Weyl discriminants gives
\[
 J_\pi(\chi_H(X))^2
 =c\,\frac{D^G(X)}{D^H(X)}
 =c\,D_H^G(X)
\]
which proves the claim.
\end{proof}

\begin{theorem}[Logarithmic Jacobi density]\label{thm:jacobi-log-density}
On the square-root Jacobi chart,
\begin{equation}\label{eq:H-jacobian-parity}
 J_H(z)=c\,R(z)
 =c'\prod_{\alpha\in\Delta}z_\alpha^{k_\alpha}
\end{equation}
Consequently, for compatible quotient measures,
\begin{equation}\label{eq:general-logarithmic-density}
 \psi^*\left(
 \frac{ds_1\cdots ds_r}{\abs{D_H^G}^{1/2}}
 \right)
 =C\prod_{\alpha\in\Delta}\frac{dz_\alpha}{\abs{z_\alpha}}
 =C'\prod_{\alpha\in\Delta}\dmeas z_\alpha
\end{equation}
Equivalently, after the finite squaring map $q_\alpha=z_\alpha^2$, the density
is a fixed multiple of product multiplicative Haar measure in the Jacobi
parameters.
\end{theorem}

\begin{proof}
By \eqref{eq:relative-discriminant-parity} and
\cref{lem:quotient-map-relative-jacobian}, there is a constant
$c\in F^\times$ such that
\[
  J_\pi(\psi(z))^2
  =
  c\left[
    \left(\prod_\alpha z_\alpha\right)R(z)
  \right]^2
\]
All root vectors may be chosen over $F$, since $G$ and $A$ are split, so both
$J_\pi(\psi(z))$ and $(\prod_\alpha z_\alpha)R(z)$ lie in
$F[z_\alpha:\alpha\in\Delta]$.  Their squares differ by a nonzero scalar.
Since this polynomial ring is a UFD, the two polynomials have the same
irreducible factors with the same multiplicities; their quotient is therefore
a unit, hence an element of $F^\times$.  Thus there is a constant
$c_0\in F^\times$ such that
\begin{equation}\label{eq:Jpi-parity}
  J_\pi(\psi(z))
  =
  c_0\left(\prod_\alpha z_\alpha\right)R(z)
\end{equation}
On the other hand, $Y(z)$ and $Q(z^2)$ are $G$-conjugate on the torus, hence
the polynomial identity
\[
 \pi\circ\psi(z)
 =\bigl(P_1(Q(z^2)),\ldots,P_r(Q(z^2))\bigr)
\]
holds everywhere.  The chain rule gives
\begin{equation}\label{eq:jacobi-chain-rule}
 J_\pi(\psi(z))J_H(z)
 =2^r\left(\prod_\alpha z_\alpha\right)J_G(z^2)
\end{equation}
Substitute \eqref{eq:Jpi-parity} and
\eqref{eq:ambient-jacobi-parity}; cancellation in the polynomial ring yields
$J_H=cR$.  Combining this with
\eqref{eq:relative-discriminant-parity} gives
\[
 \abs{D_H^G(Y(z))}^{1/2}
 =C\abs{J_H(z)}\prod_\alpha\abs{z_\alpha}
\]
which is exactly \eqref{eq:general-logarithmic-density}.
\end{proof}

\begin{remark}
The Hurwitz calculation in \cref{thm:typeCI-measure} is an independent
explicit type-$C$ realization of the same logarithmic-density mechanism as
\cref{thm:jacobi-log-density}.
\end{remark}

\subsection{Rational branches and the descent-chart theorem}

Let $\cC_G=\g/\!/G$ and let $L=F(\cQ)$.  In the equal-rank case,
\cref{cor:jacobi-fixed-quotient} makes $L$ a finite extension of
$F(\cC_G)$.  Let $\widetilde{\cC}_G$ be the normalization of $\cC_G$ in
$L$.  After shrinking the regular open set, the Jacobi parameters identify
$\widetilde{\cC}_G^\circ$ with an open subset of the smooth parameter space
$\cQ$.  We shrink once more so that this open lies in $(\Gm)^\Delta$ and
$Q(q)$ is regular semisimple throughout.  Pulling back $[2]:A\to A$ then
gives the ordinary square-root torsor, and
for $\xi\in\mathfrak B_A$ let $\cS_\xi$ be the twisted form constructed in
\cref{prop:twisted-cartan-charts}.

Over $\Fb$, every $\cS_\xi$ is the same square-root torus with coordinates
$z_\alpha$.  The logarithmic form
\begin{equation}\label{eq:log-form}
 \omega_{\log}=\bigwedge_{\alpha\in\Delta}\frac{dz_\alpha}{z_\alpha}
\end{equation}
is invariant under the $A[2]$-action $z_\alpha\mapsto\pm z_\alpha$ and hence
descends to every twisted $F$-form.  We denote the resulting density on the
$\xi$-branch by $\abs{\omega_{\log,\xi}}$.

\begin{theorem}[Jacobi descent charts]\label{thm:descent-charts}
Assume that $F$ is a non-Archimedean local field of characteristic zero and
residue characteristic different from $2$, that $G$ is split semisimple and
adjoint, that $(G,H,\theta)$ is a split symmetric pair with
$\rank H=\rank G$, and that the chosen maximal $\theta$-split torus $A$ is
$F$-split.  Then there is a dense open subset
$\cC_G^\circ\subset\cC_G$ with the following properties.
\begin{enumerate}[label=\textup{(\alph*)},leftmargin=2.3em]
\item The map
      $\widetilde{\cC}_G^\circ\to\cC_G^\circ$ is finite \'etale, and the
      Jacobi family defines a regular lift
      \[
       Q:\widetilde{\cC}_G^\circ\longrightarrow\g_{\mathrm{rs}}
\]
\item On the ordinary square-root cover, $u=t$ satisfies
      \[
       \tau(u)=a(q),\qquad Y=\Ad(t)Q\in\h
\]
      In the coordinates \eqref{eq:z-square-root-coordinates}, $Y$ is the
      linear family \eqref{eq:linear-fixed-jacobi}.
\item Put
      \[
       \mathfrak B_A=
       \ker\left(H^1(F,A[2])\longrightarrow H^1(F,H)\right)
\]
      For every $\xi\in\mathfrak B_A$, the twisted cover
      $\cS_\xi\to\widetilde{\cC}_G^\circ$ carries a regular $F$-morphism
      \[
       u_\xi:\cS_\xi\longrightarrow G,
       \qquad \tau(u_\xi)=a(q)
\]
      and hence a regular family $Y_\xi=\Ad(u_\xi)Q\in\h$.  Its $F$-points
      lie over exactly those $q$ for which $\delta(a(q))=\xi$, so the finite
      family of twisted covers exhausts the rational Cartan-lift locus.
\item On every twisted branch, the pullback of the normalized relative
      quotient density for $H^\circ$ is exactly logarithmic:
      \begin{equation}\label{eq:twisted-log-density}
       \frac{ds}{\abs{D_H^G}^{1/2}}
       =C_\xi\abs{\omega_{\log,\xi}}
\end{equation}
      for compatible local quotient measures.  For the full fixed group $H$,
      the same statement holds on sufficiently small $F$-analytic charts on
      the finite \'etale locus, with an additional locally constant
      nonzero factor coming from
      $\h/\!/H^\circ\to\h/\!/H$.  Geometrically the constant in the
      $H^\circ$ statement is independent of the twist; its displayed value
      depends only on measure normalizations.
\item On every twisted branch,
      \[
       \Ad(u_\xi^{-1})Y_\xi=Q
\]
      Although neither the twisted cover nor $u_\xi$ is asserted to extend
      across $q=0$, the inverse-conjugated expression is the pullback of the
      polynomial family $Q(q)$ on $\A^\Delta$ and therefore has the canonical
      continuation $Q(0)$, a principal nilpotent element.
\item Every semisimple descendant is again a split symmetric pair.  Let
      $L=G_X^\circ$ be the descendant of a semisimple
      $X\in\Lie(A)(F)$.  The inherited Cartan remains $F$-split and
      $\theta$-split, and the connected center $Z(L)^\circ\subset A$ is
      $\theta$-split.  The root system of $L_{\mathrm{der}}$ is
      $\Phi_X$.  On every irreducible component of $\Phi_X$ whose Weyl group
      contains $-1$, parts \textup{(a)}--\textup{(e)} apply to the
      corresponding adjoint derived factor after choosing a simple system
      for that component.  The $\theta$-split central torus is separate from this root-theoretic
      Jacobi construction, up to the usual finite central isogeny, and belongs to
      the residual centralizer-torus
      integration on the fixed side.
\end{enumerate}
For a split reductive group or a different central form, apply the statement
to the adjoint semisimple derived factor; the central eigenspaces factor off
and the additional lifting obstruction is the one in
\cref{prop:adjoint-reduction}.
\end{theorem}

\begin{proof}
For \textup{(a)}, \cref{cor:jacobi-fixed-quotient} makes
$F(\cQ)/F(\cC_G)$ a finite separable extension.  The normalization is finite
over $\cC_G$, and after restricting to the common finite \'etale locus the
birational map from the Jacobi parameter space to that normalization is an
isomorphism onto a dense open subset.  Part
\textup{(b)} is \cref{lem:theta-Q,lem:formal-descent} together with
\eqref{eq:linear-fixed-jacobi}.  Part \textup{(c)} is
\cref{prop:twisted-cartan-charts}.

For \textup{(d)}, identify a twisted branch over $\Fb$ with the ordinary
square-root chart.  Under this identification
$u_\xi=h_\xi t$ with $h_\xi\in H(\Fb)$ constant, so
$Y_\xi=\Ad(h_\xi)Y(z)$.  The relative discriminant is $H$-invariant.  The
finite component group $H/H^\circ$ acts on $\h/\!/H^\circ$ by polynomial
automorphisms; the Jacobian determinant of such an automorphism is a unit in
the polynomial coordinate ring and hence a constant, and finite order makes
that constant a root of unity.  Thus the absolute quotient density is
$H$-invariant.  Since $\omega_{\log}$ is invariant under the $A[2]$ descent
datum, \cref{thm:jacobi-log-density} descends to the $F$-form and gives
\eqref{eq:twisted-log-density} for the $H^\circ$ quotient.  Passing to the
full $H$-quotient uses the finite \'etale map discussed above; on a sufficiently
small $F$-analytic chart its Jacobian has locally constant nonzero absolute
value, giving the stated qualification.  Part \textup{(e)} follows from
\eqref{eq:twisted-inverse-conjugate} together with the fact that $Q(q)$ is
polynomial on all of $\A^\Delta$.  For part \textup{(f)}, let
$L=G_X^\circ$ with $X\in\Lie(A)(F)$.  Since $A$ centralizes $X$ and is
connected, $A\subset L$; because semisimple centralizers have the same absolute
rank as $G$, this $A$ is a maximal torus of $L$.  By
\cref{thm:descendants-split}, the descendant is again a split symmetric pair,
and by \cref{prop:root-centralizer-dimensions} its derived root system is
$\Phi_X$.  The connected center $Z(L)^\circ$ lies in the maximal torus $A$
and is therefore $\theta$-split, so it contributes no fixed Cartan rank.
For each irreducible component of $\Phi_X$ whose Weyl group contains $-1$,
pass to the corresponding adjoint derived factor.  That factor is an
equal-rank split symmetric pair with inherited $F$-split, $\theta$-split
Cartan, and parts \textup{(a)}--\textup{(e)} apply exactly as in Sections~7
and 8.  The connected $\theta$-split center is separate from these
root-theoretic Jacobi coordinates, up to the usual finite central isogeny, and
remains in the residual centralizer quotient on the fixed side.
\end{proof}

\begin{remark}
Part \textup{(c)} makes the phrase ``rational Cartan-lift branch'' literal:
each class in $\mathfrak B_A$ defines an $F$-form of the square-root cover and
an $F$-rational map to $G$.  This $H$-cohomological branch datum is distinct
from the rational-orbit problem for the limiting ambient nilpotent element.
The rational $G(F)$-orbits in the stable orbit of $Q(0)$ are governed by
\[
 \ker\left[H^1(F,G_{Q(0)})\longrightarrow H^1(F,G)\right]
\]
There may be interesting arithmetic in the specialization of Cartan-lift
branches to rational nilpotent orbits, but no canonical identification of
these two cohomology sets is asserted here.
\end{remark}

\begin{warning}
Part \textup{(f)} deliberately applies the equal-rank descent-chart theorem
to the appropriate adjoint derived factors, not to the whole reductive
descendant.  There are two distinct sources of rank defect.  First, even if
every irreducible component of $\Phi_X$ has $-1$ in its Weyl group, the
connected center $Z(G_X^\circ)^\circ\subset A$ is $\theta$-split and can make
$\rank G_X^\circ>\rank (G_X^\circ)^\theta$; these central directions belong
to the residual torus integration.  Second, root subsystems of types $A_n$
$(n\geq2)$, $D_{2m+1}$, and $E_6$ can occur inside equal-rank ambient
systems.  On such derived factors the Jacobi family maps dominantly only to
the $\theta$-fixed part of the adjoint quotient, and the corresponding rank
defect again appears in the residual centralizer torus of
\cref{eq:centralizer-disintegration}.
\end{warning}

\section{Piecewise-affine weight coordinates}
\label{sec:jacobi-weights}

\subsection{Piecewise-affine Iwasawa heights}

Fix a minimal parabolic $B=TU$ containing $A$ and a special maximal compact
subgroup $K$ for which $G(F)=U(F)T(F)K$.  We normalize
\[
 \langle\lambda,H_B(utk)\rangle=-\val(\lambda(t))
\]

\begin{lemma}[Highest-weight norms and the Iwasawa map]
\label{lem:highest-weight-iwasawa}
There are dominant characters $\lambda_1,\ldots,\lambda_s\in X^*(T)$ spanning
$X^*(T)\otimes_\mathbb Z\mathbb Q$, positive integers $N_j$, $F$-rational
representations $\rho_j:G\to\GL(V_j)$ of highest weight $N_j\lambda_j$, and
$K$-stable lattice norms $\ell_j$ such that, for a highest-weight vector
$v_j$,
\[
 \ell_j(\rho_j(g^{-1})v_j)
 =N_j\langle\lambda_j,H_B(g)\rangle
\]
\end{lemma}

\begin{proof}
Choose the representations and primitive highest-weight vectors as in the
standard highest-weight construction.  If $g=utk$, then $U$ fixes $v_j$ and
$K$ preserves the lattice norm, so
\[
 \rho_j(g^{-1})v_j
 =(N_j\lambda_j)(t)^{-1}\rho_j(k^{-1})v_j
\]
Taking valuations gives the formula.
\end{proof}

\begin{lemma}[Piecewise-affine heights on a Jacobi branch]
\label{lem:piecewise-iwasawa}
Fix $\xi\in\mathfrak B_A$ and $g_0\in G(F)$.  The descent datum changes each
geometric square-root coordinate $z_\alpha$ only by an element of $A[2]$;
its valuation is therefore well defined on the twisted branch.  After a
finite clopen refinement of $\cS_\xi(F)$ and a finite rational polyhedral
subdivision of the valuation vectors of the square-root coordinates, every component of
\[
 H_B(u_\xi(q)g_0)
\]
is a rational affine function of
\[
 \val(z_\alpha),\qquad \alpha\in\Delta
\]
Equivalently it is rational affine in the valuations of the $q_\alpha$, since
$q_\alpha=z_\alpha^2$.
\end{lemma}

\begin{proof}
Choose a finite extension $E/F$ which trivializes the twist and write over
$E$
\[
 u_\xi=h_\xi t,
 \qquad h_\xi\in H(E),
 \qquad \alpha(t)=z_\alpha^{-1}
\]
Apply \cref{lem:highest-weight-iwasawa}.  For one of the representations put
\[
 w=\rho(h_\xi^{-1})v
   =\sum_\mu w_\mu
\]
according to the $A$-weight decomposition.  Choose an $E$-lattice norm
$\ell_A$ adapted to this decomposition.  Then
\begin{equation}\label{eq:weight-adapted-minimum}
 \ell_A(\rho(t^{-1})w)
 =\min_{\mu:w_\mu\ne0}
   \left(c_\mu-\val_E(\mu(t))\right)
\end{equation}
where the $c_\mu$ are constants.  Since $G$ is adjoint, every $\mu(t)$ is a
Laurent monomial in the $z_\alpha$, so the expressions in the minimum are
affine functions of their valuations.

The norm
\[
 v'\longmapsto
 \ell(\rho(g_0^{-1})v')
\]
is another lattice norm on $V(E)$.  The difference of two lattice norms is
invariant under scalar multiplication and is locally constant on
$\mathbb P(V)(E)$.  Since this projective space is compact, it has finite
image and admits a finite clopen partition on which the difference is
constant.  Pulling that partition back along
$t\mapsto[\rho(t^{-1})w]$ shows that the norm relevant to the Iwasawa map is,
on each clopen piece, the minimum in \eqref{eq:weight-adapted-minimum} plus a
constant.  The finitely many equality hyperplanes between the affine
functions in that minimum give the required polyhedral subdivision.  Dividing
by the ramification index of $E/F$ returns the normalization of $F$-valuations.
Applying this simultaneously to the finitely many $\lambda_j$ determines
$H_B$.
\end{proof}

\begin{proposition}\label{prop:toroidal-weights}
Let $w(g,\mu)$ be a truncation weight with finitely many truncation chambers,
each defined by affine inequalities in $\mu$ and finitely many Iwasawa-height
functions of $g$, and suppose that on each chamber $w$ is polynomial in these
variables.
On every twisted Jacobi branch, after a finite clopen refinement and a finite
rational polyhedral subdivision in the joint variables
\[
 \bigl(\mu,(\val z_\alpha)_{\alpha\in\Delta}\bigr)
\]
the pullback $w(u_\xi(q)g_0,\mu)$ is polynomial in those variables on each
piece.
\end{proposition}

\begin{proof}
Apply \cref{lem:piecewise-iwasawa} simultaneously to the finitely many
minimal parabolics and height components occurring in $w$.  The original
truncation inequalities become affine inequalities in the joint variables.
Intersecting their chambers with the polyhedral chambers from the lemma and
substituting the resulting affine height functions into the chamberwise
polynomial formula gives the assertion.
\end{proof}

\begin{corollary}\label{cor:weighted-mechanism}
For every rational Cartan-lift branch the pullback of the
$H^\circ$-quotient density is the logarithmic density
$C_\xi\abs{\omega_{\log,\xi}}$; after geometric trivialization this is product
multiplicative Haar measure in the square-root Jacobi coordinates.  For the
full $H$-quotient the same statement holds locally up to the finite-quotient
factor described above.  The additional Iwasawa-height dependence introduced
by the Cartan lift is, for each fixed orbital variable and after a finite
clopen refinement, piecewise polynomial in the coordinate valuations.  No
uniform polyhedral subdivision in the orbital variable is asserted here.
\end{corollary}

The lower-dimensional faces of this polyhedral decomposition are natural
candidates for comparison with descent to proper centralizers, paralleling the
role of proper Levi subgroups in the lower-degree terms of Arthur's polynomial
weights.  Establishing such a correspondence requires the analytic descent
and neighborhood analysis deferred to the follow-up paper.

\subsection{Analytic outlook: quotient neighborhoods, induced weights, and rational-orbit covers}
Fix $\xi\in\mathfrak B_A$ and work on an $F$-analytic chart of the twisted
Cartan-lift cover $\cS_\xi$.  Write $u=u_\xi(q)$ and
$Y(q)=\Ad(u_\xi(q))Q(q)$.  Thus nontrivial Kummer classes are treated by
honest rational families rather than by pointwise choices of lifts.

The analytic role of these charts is best described without taking a weighted
germ expansion as the starting point.  Let $\Omega_\epsilon$ be a sufficiently small neighborhood of
the nilpotent point in the relevant part of $\h/\!/H$, chosen as part of the
second truncation, and let $\Omega_{\epsilon,\xi}$ denote its pullback to the
$\xi$-branch.  In the equal-rank case, \cref{thm:descent-charts} rewrites the
quotient density on this branch as a constant multiple of
$\abs{\omega_{\log,\xi}}$.  After conjugating
$Y_\xi=\Ad(u_\xi)Q$ back to $Q$, the corresponding contribution has,
schematically, the form
\begin{equation}\label{eq:induced-weight-schematic}
 \int_{\Omega_{\epsilon,\xi}}
   \abs{\omega_{\log,\xi}(z)}
   \int_{G_{Q(z^2)}(F)\backslash G(F)}
      \widehat f(g^{-1}Q(z^2)g)\,
      \overline\omega(u_\xi(z)g,\mu)\,dg 
\end{equation}
Fixed discriminant and measure normalizations are suppressed in this
schematic expression.  The family $Q(z^2)$ extends to the principal nilpotent
element $Q(0)$, while \cref{lem:piecewise-iwasawa,prop:toroidal-weights}
make the shifted truncation function, for each fixed $g$, piecewise polynomial
in the valuation variables after a finite clopen refinement and on a finite
polyhedral subdivision.  The refinement and subdivision may depend on $g$;
the uniformity required to change the order of integration belongs to the
analytic work deferred below.

The next analytic step is therefore an integration problem over the excised
quotient neighborhood, rather than the construction of a weighted germ as an
independent local object.  One must choose neighborhoods for which the
pullbacks $\Omega_{\epsilon,\xi}$ are tractable in the valuation coordinates,
justify the required changes in the order of integration, and identify the
result with weighted nilpotent orbital terms.  Schematically the induced
weight has the form
\[
 W_{\epsilon,\xi}(g,\mu)
 =\int_{\Omega_{\epsilon,\xi}}
    \overline\omega(u_\xi(z)g,\mu)\,
    \abs{\omega_{\log,\xi}(z)}
\]
with the precise domain and limiting nilpotent contribution depending on the
stratum under consideration.  A central sequence of lemmas in the follow-up paper
will characterize the relevant quotient neighborhoods and their Jacobi
pullbacks as finite unions of sets described by affine inequalities in
valuation coordinates.  Combined with the logarithmic density and the
piecewise-affine height formulas proved here, this is the mechanism expected
to produce the polynomial weights.

There remain rational-orbit questions which are logically separate from this
neighborhood integration.  Before the integration in stages of
\eqref{eq:fixed-combined-orbit}, a regular semisimple $x\in\h(F)$ has rational
$H(F)$-orbits inside its stable $H$-orbit parametrized by
\[
 \ker\left(H^1(F,H_x)\longrightarrow H^1(F,H)\right)
\]
This is the rational-orbit cover naturally attached to the original section
of $\h/\!/H$.  After passage to ambient $G$-orbits, the rational $G(F)$-orbits
in a stable regular semisimple orbit are parametrized by
\[
 \ker\left(H^1(F,G_x)\longrightarrow H^1(F,G)\right)
\]
and the rational nilpotent orbits in the stable orbit of a nilpotent element
$e$ are parametrized by
\[
 \ker\left(H^1(F,G_e)\longrightarrow H^1(F,G)\right)
\]
The finite quotient cover and the twisted Cartan-lift covers in
\cref{thm:descent-charts} serve a different but complementary purpose: they
give explicit local branches of the $H$-side section together with rational
Cartan lifts carrying those branches to a controlled ambient degeneration.
They should not be identified with either of the full rational-orbit covers.

Ordinary Shalika-germ theory remains relevant as background for the local
nilpotent singularity and its homogeneity; see Shalika \cite{Shalika} and, for
symmetric spaces, Flicker \cite{Flicker}.  Xue proves an ordinary relative
germ expansion for $\GL_n\times\GL_n\backslash\GL_{2n}$ \cite{Xue}, while
Chen proves homogeneity and parabolic descent for Arthur--Shalika germs of
weighted orbital integrals in the group case \cite{ChenWeighted}.  These
results provide useful comparisons, but the second-truncation mechanism
anticipated here runs in the opposite direction: the weight on the nilpotent
term is produced by integrating the original truncation function over a
small quotient neighborhood.

\section{The symplectic case: an explicit calculation}

Consider
\[
 (G,H)=(\Sp_{2n},\GL_n)
\]
the symmetric pair usually called type CI.  This illustrative example has
two purposes.  First, it identifies the
symplectic matrix $u_\gamma$ as a product of rank-one Cayley sections.
Second, it gives the additional Hurwitz-coordinate calculation that makes the
relative quotient measure logarithmic.  The detailed matrix identities show how the general descent mechanism can
be implemented explicitly and provide a concrete comparison with the abstract
Jacobi charts of Section~8.  The calculations of this section can also be found
in \cite{SparlingThesis}.

\subsection{The symmetric pair}
Let
\[
 J=\begin{pmatrix}0&I_n\\-I_n&0\end{pmatrix}
 \qquad
 G=\Sp_{2n}(F)
\]
Put
\[
 \varepsilon=\begin{pmatrix}0&I_n\\I_n&0\end{pmatrix}
 \qquad
 \theta(g)=\varepsilon g\varepsilon^{-1}
\]
The fixed Lie algebra is isomorphic to $\glie_n$ through
\begin{equation}\label{eq:typeCI-iota}
 \iota(X)=\frac12
 \begin{pmatrix}
 X-X^{\trans}&X+X^{\trans}\\
 X+X^{\trans}&X-X^{\trans}
 \end{pmatrix}
\end{equation}

For $X\in\glie_n$, write
\[
 \det(\lambda-X)=\lambda^n+s_1(X)\lambda^{n-1}+\cdots+s_n(X)
\]
The companion section is
\begin{equation}\label{eq:typeCI-companion}
 S(s)=
 \begin{pmatrix}
 -s_1&-s_2&\cdots&-s_n\\
 1&0&\cdots&0\\
 &\ddots&\ddots&\vdots\\
 0&&1&0
 \end{pmatrix}
\end{equation}

\subsection{Hurwitz pivots and toroidal variables}
Let $H(s)$ be the Hurwitz matrix of
\[
 \lambda^n+s_1\lambda^{n-1}+\cdots+s_n
\]
namely the matrix whose first two rows are
\[
 (s_1,s_3,s_5,\ldots),\qquad (1,s_2,s_4,\ldots)
\]
and whose subsequent rows are obtained by shifting these two rows successively
to the right; coefficients beyond $s_n$ are understood to be zero.  Let
$P_i(s)$ be the leading $i\times i$ principal minor of $H(s)$, and put
$P_0=1$.  On the open set $P_1\cdots P_n\neq0$, define
\[
 x_i=\frac{P_i}{P_{i-1}},
 \qquad
 q_1=x_1,
 \quad q_2=x_2,
 \quad q_i=\frac{x_i}{x_{i-2}}\;(i\geq3)
\]
Equivalently,
\[
 x_1=q_1,
 \qquad x_2=q_2,
 \qquad x_i=q_ix_{i-2}
\]

Define the tridiagonal matrix
\begin{equation}\label{eq:typeCI-Q}
 Q(q)=
 \begin{pmatrix}
 -q_1&-q_2&&&&\\
 1&0&-q_3&&&\\
 &1&0&-q_4&&\\
 &&\ddots&\ddots&\ddots&\\
 &&&1&0&-q_n\\
 &&&&1&0
 \end{pmatrix}
\end{equation}
Its characteristic polynomial is the continuant
\[
 \det(\lambda-Q(q))
 =\lambda^n+f_1(q)\lambda^{n-1}+\cdots+f_n(q)
\]
where $f_k(q)$ is the weighted sum of matchings covering $k$ vertices in the
path with a loop of weight $q_1$ at the first vertex and edge weights
$q_2,\ldots,q_n$.

\begin{theorem}\label{thm:typeCI-chart}
On the open set $P_1\cdots P_n\neq0$, the map $s\mapsto q$ is an isomorphism
with $\Gm^n$.  Its inverse is
\[
 s_k=f_k(q)
\]
Moreover, $S(s(q))$ and $Q(q)$ are conjugate over $F$.
\end{theorem}

Thus the Hurwitz variables give an explicit toroidal quotient chart on a big
cell of cyclic, hence regular, elements.  This cell need not consist of
semisimple elements.  Its intersection with the common regular-semisimple
locus is the domain relevant to the fixed-side Weyl expansion in Section~8.

\begin{proof}
Expanding the tridiagonal determinant gives
\[
 p_m(\lambda)=\lambda p_{m-1}(\lambda)+q_mp_{m-2}(\lambda)
\]
with $p_0=1$ and $p_1=\lambda+q_1$.  The matching polynomials satisfy the
same recurrence, so the characteristic coefficients of $Q(q)$ are the
$f_k(q)$.

We recall the corresponding Hurwitz elimination.  On the open set on which
all leading Hurwitz minors are nonzero, elimination without row exchanges has
diagonal pivots $P_i/P_{i-1}=x_i$.  Applied to a polynomial satisfying the
preceding continuant recurrence, the same elimination removes its first
continued-fraction coefficient at each step and gives
\[
 x_1=q_1,\qquad x_2=q_2,\qquad x_i=q_ix_{i-2}\quad(i\geq3)
\]
This is the finite Hurwitz--Stieltjes elimination
\cite{Gantmacher,HoltzTyaglov}; it may also be verified inductively by applying
one elimination step to the even and odd coefficient rows of the Hurwitz
matrix.  In coefficient coordinates those steps are successive shears with
unipotent derivative, so the map $s\mapsto x$ has Jacobian one.  Conversely these relations recover every
$q_i$, and the continuant recurrence recovers every coefficient $s_k=f_k(q)$.
Thus $s\mapsto q$ and $q\mapsto(f_1(q),\ldots,f_n(q))$ are inverse on the big
cell.  Finally, $Q(q)$ is cyclic, so it is conjugate over $F$ to the companion
matrix with the same characteristic polynomial.
\end{proof}

\subsection{The logarithmic measure}
The fixed-side orbital-integral application is made on the further open subset
where the characteristic polynomial has nonzero discriminant.  On that common
regular-semisimple locus, \eqref{eq:weighted-weyl-normalized} gives the quotient
density
\[
 \frac{ds}{\abs{D_H^G(\iota(S(s)))}^{1/2}}
\]
On the full Hurwitz big cell we use $D_H^G$ for the polynomial relative
determinant $\det(\ad(\iota(S(s)));\p)$, which agrees with $D^G/D^H$ on the
common regular-semisimple locus.  It satisfies
\[
 \abs{D_H^G(\iota(S(s)))}^{1/2}
 =C\prod_{i\leq j}\abs{h_i+h_j}
 =C'\abs{P_n(s)}
\]
for fixed nonzero constants $C,C'$ (with the present normalization of
$\iota$, their ratio accounts for a fixed power of $2$); here $h_i$ are the
eigenvalues of $S(s)$.  The elimination defining the
$x_i$ has Jacobian one.

\begin{theorem}\label{thm:typeCI-measure}
For compatible Haar measures,
\begin{equation}\label{eq:typeCI-measure}
 \frac{ds_1\cdots ds_n}
 {\abs{D_H^G(\iota(S(s)))}^{1/2}}
 =\frac{dx_1\cdots dx_n}{\abs{x_1\cdots x_n}}
 =\frac{dq_1\cdots dq_n}{\abs{q_1\cdots q_n}}
\end{equation}
\end{theorem}

\begin{proof}
After absorbing the fixed constants into the compatible measure
normalizations, the first equality follows from the unit Jacobian and
$P_n=x_1\cdots x_n$.  The recurrence $x_i=q_ix_{i-2}$ is a monomial automorphism of the torus whose
exponent matrix is triangular with determinant one.  It therefore preserves
product multiplicative Haar measure, giving the second equality.
\end{proof}

Thus the Hurwitz big cell is toroidal with logarithmic density; product-type
neighborhoods of its boundary become punctured polydiscs in these coordinates.
The calculation is an independent explicit type-$C$ realization of the same
logarithmic mechanism as \cref{thm:jacobi-log-density}.

\subsection{\texorpdfstring{The rank-one factorization of $u_\gamma$}{The rank-one factorization of u-gamma}}
For $\gamma\in F^\times$, set
\begin{equation}\label{eq:typeCI-D}
 D_\gamma(q)=\gamma^{-1}\diag
 \left(q_1,q_1q_2,\ldots,q_1q_2\cdots q_n\right)
 =\diag(d_1,\ldots,d_n)
\end{equation}
Define
\begin{equation}\label{eq:typeCI-u}
 u_\gamma(q)=\frac12
 \begin{pmatrix}
 I_n-D_\gamma&-D_\gamma^{-1}-I_n\\
 D_\gamma+I_n&I_n-D_\gamma^{-1}
 \end{pmatrix}
\end{equation}

Write the symplectic basis as $(v_1,\ldots,v_n,w_1,\ldots,w_n)$ and reorder it
by symplectic planes as $(v_1,w_1,\ldots,v_n,w_n)$.  In this basis,
\begin{equation}\label{eq:typeCI-block-factor}
 u_\gamma(q)
 =\bigoplus_{i=1}^n c(-d_i) 
\end{equation}
where $c(z)$ is the rank-one section \eqref{eq:cayley-section}.
The relevant roots are the pairwise strongly orthogonal long roots
$2e_1,\ldots,2e_n$ of $C_n$, so this factorization is the type-$C$ instance
of the rank-one Cartan sections of Section~5.

\begin{proposition}\label{prop:typeCI-rankone}
The matrix $u_\gamma$ is the product of the commuting root-$\SL_2$ sections
attached to $2e_1,\ldots,2e_n$.  In particular,
\begin{equation}\label{eq:typeCI-tau}
 \tau(u_\gamma(q))
 =\begin{pmatrix}-D_\gamma(q)&0\\0&-D_\gamma(q)^{-1}\end{pmatrix}
\end{equation}
\end{proposition}

\begin{proof}
The $i$th $2\times2$ block of \eqref{eq:typeCI-u} is
\[
 \frac12
 \begin{pmatrix}
 1-d_i&-d_i^{-1}-1\\
 d_i+1&1-d_i^{-1}
 \end{pmatrix}
 =c(-d_i)
\]
Apply \cref{lem:sl2-section} in each symplectic plane.  The long-root
subgroups commute, so the Cartan map factors componentwise.
\end{proof}

The identities
\[
 \frac{d_i}{d_{i-1}}=q_i\qquad(i\geq2)
\]
show that the Hurwitz variables are successive ratios of the strongly
orthogonal coroot coordinates.  The parameter $\gamma$ supplies the common
square-class normalization invisible to those ratios.  It labels the
square-class variant of this explicit lift and the resulting rational regular
nilpotent orbit in the worked example.

\subsection{The nilpotent limit}
Let $E_{11}$ be the first matrix unit and define
\begin{equation}\label{eq:typeCI-Qgamma}
 Q_\gamma(q)=
 \begin{pmatrix}
 Q(q)+q_1E_{11}&\gamma E_{11}\\
 \gamma^{-1}q_1^2E_{11}&-Q(q)^{\trans}-q_1E_{11}
 \end{pmatrix}
\end{equation}
A direct block calculation gives
\begin{equation}\label{eq:typeCI-conjugation}
 \Ad(u_\gamma(q)^{-1})\iota(Q(q))=Q_\gamma(q)
\end{equation}
The right side is polynomial in $q$ and extends to $q=0$, exactly mirroring
the inverse-conjugated extension in \cref{thm:descent-charts}.  Its value is
\[
 e_\gamma=
 \begin{pmatrix}
 N&\gamma E_{11}\\0&-N^{\trans}
 \end{pmatrix},
 \qquad
 N=\sum_{i=2}^nE_{i,i-1}
\]
Starting from $w_n$, repeated application of $e_\gamma$ produces one Jordan
chain of length $2n$, so $e_\gamma$ is regular nilpotent in $\splie_{2n}$.
Changing $\gamma$ by a square does not change its rational orbit.

\begin{theorem}[Worked symplectic example]\label{thm:typeCI-complete}
For the illustrative symplectic pair, the explicit data
\[
 (q,Q,u_\gamma,Q_\gamma,e_\gamma)
\]
have the following simultaneous properties:
\begin{enumerate}[label=\textup{(\roman*)},leftmargin=2.3em]
\item $Q(q)$ represents the companion-section orbit over the invariant point
      $s(q)$;
\item the relative quotient density is $\prod_i\dmeas q_i$;
\item $\tau(u_\gamma(q))$ lies in a fixed maximal $\theta$-split torus;
\item $\Ad(u_\gamma(q)^{-1})\iota(Q(q))$ extends polynomially to $q=0$;
\item its value at $q=0$ is the rational regular nilpotent element
      $e_\gamma$.
\end{enumerate}
Moreover, the Cartan lift in (iii) is exactly the product of the universal
rank-one sections associated with the strongly orthogonal long roots of
$C_n$.
\end{theorem}

The theorem gives an independent worked realization of the general descent
mechanism proved above.

\subsection{What the chart explains in the fixed-side integral}
Substituting the Hurwitz variables into the normalized form
\eqref{eq:weighted-weyl-normalized} changes the quotient measure to
\[
 \prod_{i=1}^n\dmeas q_i
\]
The family $Q_\gamma(q)$ has a regular nilpotent endpoint, while
$\tau(u_\gamma(q))$ is diagonal with entries monomial in the successive
products $q_1\cdots q_i$.  Thus the two difficulties identified in
Section~\ref{sec:weighted-origin} are resolved simultaneously:
\begin{enumerate}[label=\textup{(\roman*)},leftmargin=2.3em]
\item the nilpotent limit is expressed by a single regular conjugated family;
\item for each fixed orbital variable, after a finite clopen refinement, every
      Cartan height entering the shifted truncation function is piecewise
      affine in the valuations of the $q_i$.
\end{enumerate}
Consequently, after a small quotient neighborhood is pulled back to the
Hurwitz chart, both its measure and the truncation function to be integrated
over it are expressed in valuation coordinates.  This is the explicit model
for the general mechanism described in
\cref{prop:toroidal-weights,eq:induced-weight-schematic}.

\end{document}